\documentclass{aptpub}
\usepackage{xspace}
\usepackage{subcaption}
\usepackage{graphicx}

\authornames{S.B. Connor and W.S. Kendall} 
\shorttitle{Perfect Simulation of $M/G/c$ Queues} 

 \newtheorem{algorithm}{Algorithm}

 \newcommand{\Expect}[1]{\operatorname{\mathbb{E}}\left[#1\right]}
 \newcommand{\Indicator}[1]{\operatorname{\mathbb{I}}\left[#1\right]}

 \newcommand{\Prob}[1]{\operatorname{\mathbb{P}}\left[#1\right]}

 \newcommand{\wt}{\widetilde{t}}
 \newcommand{\tleq}{\trianglelefteq}
 \newcommand{\wD}{\widetilde{D}}
 \newcommand{\wJ}{\widetilde{J}}
 \newcommand{\wS}{\widetilde{S}}
 \newcommand{\wX}{\widetilde{X}}

\begin{document}

\title{Perfect Simulation of $M/G/\lowercase{c}$ Queues} 

\authorone[University of York]{Stephen B. Connor} 
\authortwo[University of Warwick]{Wilfrid S. Kendall} 

\addressone{Department of Mathematics, University of York, York, YO10 5DD, UK.} 
\addresstwo{Department of Statistics, University of Warwick, Coventry, CV5 6FQ, UK.} 

\begin{abstract}
In this paper we describe a perfect simulation algorithm for the stable $M/G/c$ queue.
 Sigman (2011: Exact Simulation of the Stationary Distribution of the FIFO M/G/c Queue. Journal of Applied Probability, 48A, 209--213)
 showed how to build a dominated CFTP algorithm for perfect simulation of the super-stable $M/G/c$ queue operating under First Come First Served discipline.
 Sigman's method used a
 dominating process provided by the corresponding $M/G/1$ queue (using Wolff's sample path monotonicity, which applies when service durations are coupled in order of initiation of service).
 The method
 exploited
 the fact that the workload process for the $M/G/1$ queue remains the same under different queueing disciplines,
 in particular under the Processor Sharing discipline, for which a dynamic reversibility property holds.
 We generalize Sigman's construction to the stable case by comparing the $M/G/c$ queue to a copy run under Random Assignment. 
 This allows us to produce a na{\"\i}ve perfect simulation algorithm based on running the dominating process back to the time it first empties.
 We also construct a more efficient algorithm that uses sandwiching by lower and upper processes constructed as 
 coupled 
 $M/G/c$ queues 
 started respectively from the empty state and the state of the $M/G/c$ queue under Random Assignment.
 A careful analysis shows that appropriate ordering relationships can still be maintained, so long as service durations 
 continue to be coupled in order of initiation of service.
 We summarize statistical checks of simulation output, 
 and demonstrate that the mean run-time is finite so long as the second moment of the service duration distribution is finite.
\end{abstract}





\section{Introduction}\label{sec:introduction}

Coupling methods for queues have a celebrated history, stretching back to the seminal work of Loynes \cite{Loynes-1962}, 
who discussed stability results for very general queues using what would today be described as coupling comparisons, 
together with recursive formulations of queueing dynamics using queues whose commencements originate further and further back in the past. 
More recently Wolff \cite{Wolff-1987} (correcting  \cite{Wolff-1977}) showed how to establish comparisons between different queueing service disciplines for a single queue with multiple servers. 
Here the coupling argument involves assigning the \(n^\text{th}\) service duration to the \(n^\text{th}\) initiation of customer service (as opposed to the \(n^\text{th}\) customer arrival time).

The work of  \cite{Loynes-1962}, together with subsequent work on stochastically recursive systems and renovating events (for example  \cite{BorovkovFoss-1992}) is now viewed as one of the precursors of the celebrated Coupling from the Past (CFTP) algorithm of Propp and Wilson \cite{ProppWilson-1996a}, the first of a series of practical algorithms for performing exact or perfect simulation. 
It was thus a natural step from the first CFTP algorithms to consider how to apply CFTP to the problem of simulating from the equilibria of queueing systems: a variety of CFTP techniques applied to various queues of finite capacity are discussed in  \cite{MurdochTakahara-2006}.
The finite capacity requirement appeared to be an inevitable constraint, imposed by the nature of the original algorithm \cite{ProppWilson-1996a}, which uses maximal and minimal elements of the state-space to generate upper and lower bounding processes.
Clearly a queue without finite capacity will not in general possess a state-space with maximal element.
However the CFTP idea need not be limited to bounded processes. 
Working in the context of stochastic geometry, it has been shown how to replace the maximal element by a suitable dominating random process, so as to produce a \emph{dominated CFTP} algorithm \cite{Kendall-1998d,KendallMoeller-2000}.
Indeed, Kendall \cite{Kendall-2004c} has shown how in principle dominated CFTP can be applied to \emph{any} regular geometrically ergodic Markov chain, using Foster-Lyapunov criteria, small set regeneration, and a dominating process which is in fact a \(D/M/1\) queue.
(Extensions to a class of non-geometrically ergodic chains are given in  \cite{ConnorKendall-2007a}.) 
The work in \cite{Kendall-2004c} strongly indicates that one should search for \emph{practical} dominated CFTP algorithms which solve the problem of simulating from the equilibria of geometrically ergodic queues. 
Very recently Sigman \cite{Sigman-2011} has shown how to use dominated CFTP in the so-called ``super-stable'' case of the $M/G/c$ queue with First Come First Served (\(FCFS\)) discipline.
Here ``super-stable'' means that arrival and service rates are such that the queue remains stable even if \(c-1\) of the servers are removed.

Sigman's work
uses the fact that the \(M/G/c\) queue is dominated by an \(M/G/1\) queue under \(FCFS\) discipline (equivalent in this single-server case to First In First Out or \(FIFO\) discipline), which itself is stable by virtue of super-stability of the original queue;
the trajectory of an equilibrium instance of the dominating queue can be generated from times in the arbitrarily distant past using the observation that the queue workload does not depend on service discipline,
and therefore these trajectories can be reconstructed from a dynamically reversible relative which uses a Processor Sharing (PS) discipline. Thus the strategy is:
 \begin{enumerate}
  \item Given the same inputs, the total workload of a super-stable $M/G/c$ queue is sample-wise dominated by the total workload of an $M/G/1$ queue;
\item Using the processor sharing discipline, an $M/G/1\; [PS]$ queue is dynamically reversible in time (and has the same workload process as the corresponding $M/G/1\; [FIFO]$ queue);
\item Thus the $M/G/1\; [PS]$ queue can be used as a dominating process;
\item Coalescence occurs when the $M/G/1\; [PS]$ queue empties;
\item The workload process of the $M/G/1\; [PS]$ queue can be decoded to generate the arrival times and service durations of the underlying arrival process.
 \end{enumerate}
The second moment of the service duration must be finite if the algorithm is to have finite mean run-time.

Sigman's approach \cite{Sigman-2011} is limited to the super-stable case by the need to dominate the target queue using a simple \(M/G/1\) queue. 
A different approach \cite{Sigman-2012} uses regenerative techniques to extend to the merely stable case, but this different approach unfortunately results in run-times of infinite mean \cite{Sigman-2013}.
In the present paper we show how to generalize the approach of \cite{Sigman-2011} to deal with stable \(M/G/1\) queues. 
The essence of the idea is to replace the dominating \(M/G/1\) queue by an \(M/G/c\) queue run with random assignment (\(RA\)) 
(this may also be viewed as an independent collection of \(c\) different \(M/G/1\) queues; hence we write it indifferently as \(M/G/c\; [RA]=[M/G/1]^c\));
and then to use  Wolff's observation \cite{Wolff-1977} that sample-wise monotonicity between \(M/G/c\; [FCFS]\) and \([M/G/1]^c\) can be arranged if service durations are assigned in order of initiation of service.
The idea is simple enough: however considerable care needs to be exercised in order to ensure that the dominating process really does dominate the target chain in an appropriate sense. (This idea, which forms the basis of our Algorithm~\ref{alg:1}, was proposed independently in~\cite{BlanchetWallwater-2014}, although that paper does not contain a proof of the algorithm's correctness.) Moreover, in order to achieve smaller run-times by using a refined algorithm,
it is necessary to show that
appropriate ordering (or \emph{sandwiching}) relationships are maintained between upper and lower processes started at different initial times \(-T\). 
The attraction of this extension to  Sigman's work \cite{Sigman-2011} is that it allows simulation methods to be applied precisely in the case when \(M/G/c\) queues will be most relevant, namely when using a single server (\(c=1\)) would result in loss of stability.

It is appropriate here to mention some further related papers on perfect simulation and queueing. 
 Fernandez et al. \cite{FernandezFerrariGarcia-2002} study exclusion models via ensembles of Peierls contours in a spatial problem as a kind of spatially distributed loss network; however the methods are specific to loss networks with Poisson inputs and Exponential lifetimes.
Blanchet and Dong \cite{BlanchetDong-2013} apply dominated CFTP to a $GI/GI/c/c$ \emph{loss} process, using a $GI/GI/\infty$ queue as dominating process. Rather than waiting until the dominating process empties (a time which generally grows exponentially in the arrival rate), they look for a time interval $[a,b]$ for which all customers present at time $a$ have departed by time $b$, and over the entirety of which the infinite server system has less than $c$ customers. The two processes will have coalesced by time $b$, and so coalescence is determined by watching the dominating process alone. 
The authors overcome a significant technical difficulty in \cite{BlanchetDong-2013} by showing how to simulate a renewal process input in reverse time; however their method for coupling target and dominating processes involves truncation of an infinite server system, and this cannot be applied in our \(M/G/c\) context.
Mousavi and Glynn \cite{MousaviGlynn-2013} discuss perfect simulation for reflected Brownian motion in a wedge (consequently gaining information on heavy traffic approximation for queues). 
Attention is focused on stochastic differential equation problems, and links are made with the approach of \cite{BeskosRoberts-2005a} to exact simulation for solutions to stochastic differential equations. Blanchet and Chen~\cite{BlanchetChen-2014} show how to perform perfect simulation for the workload vector of a network of $d$ queueing stations with Poisson inputs, in which each arrival brings a vector of service times describing the additional work to be carried out at each station. This requires the existence of a finite moment generating function for the vector of service times, although this is subsequently relaxed in~\cite{BlanchetWallwater-2014}.

We conclude this introductory section by setting out the plan of the paper.
Section \ref{sec:dynamic} reviews notation and fundamental facts for our target queue \(M/G/c\) and the intended dominating process \([M/G/1]^c\).
Section \ref{sec:domination} describes extensions (Theorems \ref{thm:path-wise-domination1}, \ref{thm:path-wise-domination2}) of a classical domination result from queueing theory (Theorem \ref{thm:classical-domination}),
which prove the queue comparisons necessary to establish the required domination relationships.
Section \ref{sec:domCFTP} provides a proof of a simple dominated CFTP algorithm (Algorithm \ref{alg:1}) based on the regeneration which happens when the dominating process \([M/G/1]^c\) empties. 
However the run-time for this algorithm will be large in cases when the target process \(M/G/c\; [FCFS]\) rarely empties, which is precisely the set of circumstances for which multi-server queues have practical utility!
Section \ref{sec:sandwich} describes and proves the validity of a refined algorithm (Algorithm \ref{alg:2}), based on the sandwiching of the target process between pairs of upper and lower processes themselves generated from the dominating process: the regeneration used in the simple algorithm is replaced by consideration of when these upper and lower processes agree at time zero. 
At the price of increased complexity (proving domination relationships hold not just between lower, target, upper and dominating process, but also between different pairs of upper and lower process),
the algorithm run-time can be substantially decreased.
Empirical demonstrations of the savings which can be obtained, as well as the correctness of the algorithm in the computable \(M/M/c\) case, are demonstrated by representative simulations in Section \ref{sec:mmc}.
Finally, Section \ref{sec:conclusion} discusses further research possibilities.


\section{Dynamic reversibility}\label{sec:dynamic}

In this section we work with the \(M_\lambda/G/c\) queue with arrival rate \(\lambda\) under two different allocation rules: 
\emph{first-come-first-served}
(\(M_\lambda/G/c\; [FCFS]\))
and \emph{random assignment} (\(M_\lambda/G/c\; [RA]\)); in this second case each arrival is assigned randomly and independently to
one of the \(c\) servers without regard to load on each server. 
The \(M_\lambda/G/c\; [RA]\) queue will serve as a dominating process for the dominated CFTP algorithm 
to be described in later sections. 
The queue \(M_\lambda/G/c\; [RA]\) may be viewed as a system of \(c\) independent \(M_{\lambda/c}/G/1\; [FCFS]\) queues, each with arrival 
rate \(\lambda/c\). 
To emphasize this, we sometimes write \(M_\lambda/G/c\; [RA]\) as \([M_{\lambda/c}/G/1]^c\).

We follow the notation of \cite{Sigman-2011} and \cite{Asmussen-2003}. Firstly, we consider a general $\cdot/\cdot/c\; [FCFS]$ queue and review the  Kiefer-Wolfowitz construction of a workload vector \cite{KieferWolfowitz-1955}.
Let ${\mathbf V}(t) = (V(1,t),V(2,t),\dots, V(c,t))$ denote the workload vector at time \(t\geq0\). 
To be explicit, the $V(1,t)\leq V(2,t) \leq\dots$ represent the ordered amounts of \emph{residual} work in the system for the $c$ servers at time $t$, bearing in mind the FCFS queueing discipline. Customer \(n\) arrives at time $t_n$ (for \(0\leq t_1\leq t_2\leq \ldots\)). Inter-arrival times are $T_n = t_{n+1}-t_n$ (where we set $t_0 = 0$). 
Observing $\mathbf V$ just before arrival of the \(n^\text{th}\) customer (but definitely after the arrival of the \((n-1)^\text{st}\) customer) generates the process $\mathbf W_n$: in case \(t_{n-1}<t_n\) we have
$\mathbf W_n= \mathbf V(t_n-)$. This satisfies the
Kiefer-Wolfowitz recursion
\begin{equation}\label{eqn:KW-recursion}
\mathbf W_{n+1} \quad=\quad R(\mathbf W_n + S_n\mathbf e - T_n\mathbf f)^+ ,\quad \text{ for }n\geq 0\,, 
\end{equation}
where
\begin{itemize}
 \item \(\mathbf W_n + S_n\mathbf e\) adds \(S_n\) to the first coordinate only of the vector \(\mathbf{W_n}\), 
 \item \(\mathbf W_n + S_n\mathbf e - T_n\mathbf f\) subtracts \(T_n\) from each of the coordinates of \(\mathbf W_n + S_n\mathbf e\),
 \item \(R(\mathbf W_n + S_n\mathbf e - T_n\mathbf f)\) reorders the coordinates of the vector \(\mathbf W_n + S_n\mathbf e - T_n\mathbf f\) in increasing order,
 \item and \(R(\mathbf W_n + S_n\mathbf e - T_n\mathbf f)^+\) replaces negative coordinates of \(R(\mathbf W_n + S_n\mathbf e - T_n\mathbf f)\) by zeros. 
\end{itemize}
Since each of these operations is a coordinate-wise monotonic function of the previous workload vector \(\mathbf{W}_n\) and the service duration \(S_n\), an argument from recursion 
shows that the coordinates of \(\mathbf{W}_n\) depend monotonically on the initial workload vector and the sequence of service durations, once the arrival time sequence is fixed. See also, for example, remarks in  \cite{Moyal-2013} on Join Shortest Workload (\(JSW\)) disciplines for systems of parallel \(FIFO\) queues -- corresponding to \(\cdot/\cdot/c\; [FCFS]\). If \(t_n\leq t<t_{n+1}\) then we obtain \(\mathbf V(t)\) from \(\mathbf{W}_n\) by subtracting \(t-t_n\) from all the workload components and then taking positive parts: 
\begin{equation}\label{eqn:KW-recursion-cts}
{\mathbf V}(t) \quad=\quad ({\mathbf W}_n - (t-t_n)\mathbf f)^+ \,.
\end{equation}
Arguing as before, the coordinates of \(\mathbf{V}(t)\) depend monotonically on the initial workload vector and the sequence of service durations, once the arrival time sequence is fixed.

We are specifically interested in the $M_\lambda/G/c\; [FCFS]$ queue with arrival rate $\lambda$ and independent and identically distributed service durations $S_n$. 
Let \(G\) be the common distribution of the \(S_n\), and set $\Expect S = 1/\mu$. We shall assume throughout that $\Expect{S^2}<\infty$, in order to guarantee finite mean run-time of our algorithms (as detailed in Section~\ref{sec:mmc}). Write $\rho = \lambda/\mu$; we consider the stable case \(\rho<c\). We will compare this to the \([M_{\lambda/c}/G/1]^c\) system with total arrival rate \(\lambda\) and service durations as above.
That is, rather than operating under FCFS, we assign incoming customers to one of $c$ independent $M/G/1$ queues uniformly at random. Each of these queues sees arrivals at rate $\lambda/c$ and therefore has sub-critical traffic intensity \(\lambda/(c\mu)\). 
As noted in \cite{Sigman-2011}, it is a classical fact from queueing theory that the workload of an individual $M_{\lambda/c}/G/1$ queue is invariant under changes of work-conserving discipline. 
We can exploit this by using the \emph{processor sharing} discipline (PS),
since under this discipline the single-server queue workload vector process can be viewed as dynamically reversible \cite[Section 5.7.3]{Ross-1996}. This means that
the reverse process is a system of the same type, with customers again arriving at a Poisson rate $\lambda/c$, 
and with workloads having the same distribution $G$ as $S_n$, 
but with the state now representing the amount of work \emph{already performed} on customers still in the system. Since each of the $c$ independent copies of $M_{\lambda/c}/G/1$ is dynamically reversible under PS, it follows that the \(M_\lambda/G/c\;[RA]=[M_{\lambda/c}/G/1]^c\) queue is itself dynamically reversible under PS applied to each component queue.


\section[Domination of M/G/c]{Domination of $M_\lambda/G/c$}\label{sec:domination}

In this section we develop results based on the observation  \cite{Wolff-1987} that 
it is possible to arrange for the \(M_\lambda/G/c\) queue to be \emph{path-wise} dominated by 
\(c\)-server queues using other queueing disciplines,
if the two queues are coupled by listing initiations of service in order and assigning the same service duration to the \(n^\text{th}\) initiation
of service in each queue. 
(As noted below, this assignation in order of initiation of service is crucial.)
The fundamental idea is to establish that the non-FCFS system completes less total work by any fixed time, since 
corresponding services initiate later (when listed in order of initiation as above). 

Let $Q_t$ denote the queue length at time $t$, and write $|\mathbf V(t)| = V(1,t) + \dots + V(c,t)$ for the total workload (remaining work) at time $t$.
We begin by citing a classic result proved in queueing theory monographs.
\begin{thm}{\cite[Chapter XII]{Asmussen-2003}}\label{thm:classical-domination}
We consider an \(M_\lambda/G/c\) queueing system under various queueing disciplines.
We use
\(\leq_\textup{so}\) to refer to stochastic ordering of distribution functions, and use
tildes to refer to quantities pertaining to the system when it evolves under a possibly non-FCFS allocation rule; 
unadorned quantities pertain to the system when it evolves under an FCFS
allocation rule. For any (possibly non-FCFS) allocation rule, it holds for initially empty systems that 
\[ 
Q_t \leq_{\textup{so}} \tilde Q_t \qquad\text{  and  }\qquad |\mathbf V(t)| \leq_{\textup{so}} |\tilde{\mathbf V}(t)| \qquad \text{  for all $t\geq 0$.} 
\]
Similarly, $|\mathbf W_n|  \leq_{\textup{so}}  
|\tilde{\mathbf W}_n|$ for all $n$.
\end{thm}
Note that the concept of a Kiefer-Wolfowitz workload vector is not well-defined for general non-FCFS queues: nevertheless the total amount of work \(|\tilde{\textbf{V}}(t)|\), respectively \(|\tilde{\textbf{W}}_n|\), can be defined unambiguously as the total amount of residual work currently in the system. (Indeed, given an arrival at time $t_n$ we can generate a service duration $S_n$ and increase the total workload by this amount: $|\mathbf V(t_n)| = |\mathbf V(t_n-)| + S_n$. Speaking algorithmically, this service duration is then stored ready for use at the time of the $n^{\text{th}}$ initiation of service; in non-FCFS queues this will not typically be the time at which the $n^{\text{th}}$ customer to arrive commences service.)

It follows immediately that the queue length and residual workload of the $M/G/c$ queue under FCFS 
are stochastically dominated by those of the same queue with any alternative allocation rule. 
In fact this result generalizes to general $GI/G/c$ queues. 
However the result does {\bf not} carry over to domination in the sense of sample paths
if the corresponding coupling assigns the same service duration to the same individual (where ``same'' means same in order of arrival); see Wolff's correction \cite{Wolff-1987}
of \cite{Wolff-1977}.
To establish such a domination, one has to take some care to link service durations between the two different systems in the right way, namely, to ensure
that the same service duration is assigned to the \(n^\text{th}\) initiation
of service in each queue. For the purposes of our dominated CFTP argument, we need to generalize this result to cases when the allocation rule may change
at some fixed time, and also to certain cases where each of an initial subsequence of service durations is reduced to zero 
(this device allows us to include cases
in which one of the systems is not empty at time zero).

The argument given below is a modest extension of that of Asmussen \cite[Chapter XII]{Asmussen-2003}, but is central to the arguments of later sections
of the current paper.
\begin{thm}\label{thm:path-wise-domination1}
Consider a FCFS \(c\)-server queueing system viewed as a function of (a) the sequence of arrival times,
(customers arriving at times \(0\leq t_1\leq t_2\leq t_3\leq \ldots\))
and (b) the sequence of service durations \(S_1, S_2, S_3, \ldots\) assigned \emph{in order of initiation of service}
(positive except for a possible initial subsequence of zeros).
Then this system depends monotonically on the inputs \(0\leq t_1\leq t_2\leq t_3\leq \ldots\) and service durations
\(S_1, S_2, S_3, \ldots\), in the sense that for each \(m\) the \(m^\text{th}\) initiation of service \(J_m\) and the \(m^\text{th}\)
time of departure \(D_m\) are increasing functions of these inputs. Moreover, if the arrival times are fixed then for each \(t\geq0\) the Kiefer-Wolfowitz workload vector \({\mathbf V}(t)\)
(considered coordinate-by-coordinate) depends monotonically on the initial workload vector (measured immediately after time \(0\)) and the sequence of service durations corresponding to non-zero arrival times.
\end{thm}

\begin{proof}
As noted in the discussion of Equations \eqref{eqn:KW-recursion} and \eqref{eqn:KW-recursion-cts} above, the coordinates of the Kiefer-Wolfowitz workload vector depend monotonically on the initial workload vector and the sequence of service durations once the sequence of arrival times is fixed. This settles the second part of the theorem. It remains to prove the first part.

Let \(\wt_m\geq t_m\), \(\wS_m\geq S_m\), \(\wJ_m\), \(\wD_m\) be the time of arrival, the service duration, the time of initiation of service
and the time of departure for the \(m^\text{th}\) individual of the comparison FCFS system. 
We need to show that \(\wJ_m\geq J_m\) and \(\wD_m\geq D_m\) for all \(m\).
We have stipulated that
the two sequences of service durations \(\{S_m:m\geq1\}\) and \(\{\wS_m:m\geq1\}\) contain positive service 
durations except for initial subsequences of zeros. 

We will use an inductive proof, and prepare for this by establishing useful representations of \(D_m\), \(\wD_m\)  and \(J_{m+c}\), \(\wJ_{m+c}\).
We concentrate on \(D_m\), \(J_{m+c}\) \emph{etc}, for simplicity.
First note for \emph{any} allocation policy
\begin{equation}\label{eq:key}
 J_{m+c} \quad\geq\quad \max\{t_{m+c}, D_{m}\}\,,
\end{equation}
with equality holding in the case when the FCFS policy applies,
since the $(m+c)$th service starts either when the $(m+c)$th customer arrives (if there
is a spare server, which is to say when \(D_m\leq t_{m+c}\)), or if not then service initiates exactly when the relevant departure frees up a server: this happens at time $D_{m}$ in the case of the FCFS policy, and otherwise can happen no earlier.

On the other hand
the \(m^\text{th}\) departure time \(D_m\) is given by
\begin{equation}
   D_m \quad=\quad \min{}^{(m)}\{  J_1+  S_1,   J_2+  S_2,   J_3+  S_3,\dots \}\,,
\label{eq:departure0}
\end{equation}
where $\min{}^{(m)}$ denotes the operation that returns the $m^\text{th}$ order statistic. We now refine this
so as to involve only finitely many times of completions of service on the right-hand side.

First note that only at most \(c\) customers can actually be in service at time \(J_{m+c}\). Therefore
\begin{equation}\label{eq:basic}
D_m \quad\leq\quad J_{m+c}\,.
\end{equation} 
It follows from this that
\begin{equation}\label{eq:cap}
 D_m \quad\leq\quad J_{m+c} \quad\leq\quad J_{m+c+r} \quad\leq\quad J_{m+c+r}+S_{m+c+r} \qquad \text{ whenever } r \geq0\,.
\end{equation}
If \(S_{m+c}>0\) then also \(S_{m+c+r}>0\) for \(r\geq0\), so in this case \(D_m\leq J_{m+c+r}<J_{m+c+r}+S_{m+c+r}\) for \(r\geq0\).
Thus in this case we can improve on \eqref{eq:departure0} and write \(D_m\) in terms of an order statistic over a specific finite population:
\begin{equation}\label{eq:departure}
  D_m \quad=\quad \min{}^{(m)}\{J_1+S_1, J_2+S_2, J_3+S_3,\dots, J_{m+c-1}+S_{m+c-1}\}\,.
\end{equation} 
On the other hand, if \(S_{m+c}=0\) then also \(S_1=S_2=S_3=\ldots=S_{m+c}=0\). In that case service is immediate on arrival, 
so \(t_m=J_m=J_m+S_m=D_m\), and so \eqref{eq:departure} still holds (noting that there must be at least one server, so \(c\geq1\)).

Consider the inductive hypothesis that \(\wD_u\geq D_u\) for \(u=1,\ldots,m-1\) and \(\wJ_v\geq J_v\) for \(v=1,\ldots,m+c-1\).
This holds for \(m=1\), since under FCFS the first \(c\) people are served at their arrival times, so \(\wD_u=\wt_u+\wS_u\geq t_u+S_u\) for \(u=1, \ldots, c\). Suppose the inductive hypothesis holds for \(m=n\). Then we can apply the monotonic formulae \eqref{eq:key}, \eqref{eq:departure}
and deduce that the inductive hypothesis holds for case \(m=n+1\) too. Thus the first part of the theorem follows by mathematical induction.
\end{proof}

Consider two instances of \(M/G/c\;[FCFS]\), coupled monotonically using the construction implied in Theorem \ref{thm:path-wise-domination1}, based on the same sequence of arrival times, 
using sequences of service durations that agree once arrival times become positive, and such that the Kiefer-Wolfowitz workload vector of one strictly dominates that of the other at time \(0+\).
We can remark that the queues can couple
successfully (which is to say, attain the same state at the same time, called the \emph{coupling time})
only at a time when both instances have idle servers.
For the monotonicity implies that one has total workload strictly larger than the other up to the time when they first couple
successfully.
Since arrival times are fixed and shared by both systems, 
successful 
coupling of the two processes cannot occur at the time of an arrival (which simply increases the workload by equal amounts for each queue).
On the other hand, if both queues have \(c\) or more individuals in the system then the workloads decrease at the same rate.
It follows that 
successful
coupling will occur at a time when (a) an arrival does not happen, (b) there are strictly fewer than \(c\) individuals in the smaller system (hence, the smaller system has an idle server).
The coupling implies that at the coupling time the same will be true of the larger system (that is, it too will have an idle server).

\begin{thm}\label{thm:path-wise-domination2}
Consider a \(c\)-server queueing system viewed as a function of (a) the sequence of arrival times,
(customers arriving at times \(0\leq t_1\leq t_2\leq t_3\leq \ldots\))
and (b) the sequence of service durations \(S_1, S_2, S_3, \ldots\) assigned \emph{in order of initiation of service}
(positive except for a possible initial subsequence of zeros).
Consider the following cases of different allocation rules, 
in some cases varying over time:  
\begin{enumerate}
\item $\cdot/\cdot/c\;[RA]$;
\item $\cdot/\cdot/c\; [RA]$ until a specified non-random time $T$, then switching to $\cdot/\cdot/c\; [FCFS]$;
\item $\cdot/\cdot/c\; [RA]$ until a specified non-random time \(T'\), $0\leq T'\leq T$, then switching to $\cdot/\cdot/c\; [FCFS]$;
\item $\cdot/\cdot/c\; [FCFS]$;
\end{enumerate}
For the sake of an explicit construction,
when initiations of service tie then we break the ties using order of arrival time. 
On change of allocation rule to FCFS, 
customers in system but not yet being served are placed at the front of the queue in order of arrival-time; service initiates immediately for the appropriate number of customers if there are servers free.
If all this holds then case \(k\) dominates case \(k+1\), in the sense that the \(m^\text{th}\) initiation of service in case \(k+1\) occurs no later than the \(m^\text{th}\)
initiation of service in case \(k\), and the \(m^\text{th}\) departure in case \(k+1\) occurs no later than the \(m^\text{th}\)
departure in case \(k\). 

Moreover, for all times $t\geq T'$ the Kiefer-Wolfowitz workload vector for case 3 dominates (coordinate-by-coordinate) that of case 4, with similar domination holding for cases 2 and 3 for all $t\geq T$.
\end{thm}

\begin{proof}
First observe that the desired relationships between case \(1\) and case \(2\), and between case \(2\) and case \(3\), follow immediately
once we have established the desired relationship between case \(3\) and case \(4\). 
For the two compared systems evolve in exactly the same way up to time \(T\), respectively \(T'\), and so we may simply argue in terms of the processes started at time \(T\), respectively \(T'\),
for example in case \(1\) adjusting the sequence of arrival times by \(t_n\mapsto \min\{t_n-T,0\}\) and replacing service durations of services initiated before \(T\) by the residual service duration at \(T\).

The argument therefore depends on the comparison between case \(3\) and the FCFS case \(4\). Letting quantities with tildes refer to case \(3\), and using the notation of the proof of Theorem \ref{thm:path-wise-domination1}, we find from the arguments for \eqref{eq:departure} that
\begin{align*}
 \wD_m \quad&=\quad \min{}^{(m)}\{\wJ_1+S_1, \wJ_2+S_2, \wJ_3+S_3,\dots, \wJ_{m+c-1}+S_{m+c-1}\} \,,\\
 D_m   \quad&=\quad \min{}^{(m)}\{J_1+S_1, J_2+S_2, J_3+S_3,\dots, J_{m+c-1}+S_{m+c-1}\}\,,
\end{align*}
(noting that here the service durations agree when considered in order of initiation) 
and from the arguments for \eqref{eq:key}, and the fact that FCFS holds for case \(4\), that
\begin{align*}
 \wJ_{m+c} \quad&\geq\quad \max\{t_{m+c}, \wD_{m}\}\,,\\
 J_{m+c} \quad&=\quad \max\{t_{m+c}, D_{m}\}\,.
\end{align*}

Consider the inductive hypothesis that \(\wD_u\geq D_u\) for \(u=1,\ldots,m-1\) and \(\wJ_v\geq J_v\) for \(v=1,\ldots,m+c-1\).
This holds for \(m=1\), since under FCFS the first \(c\) people are served at their arrival times while 
service cannot occur earlier under any other allocation policy, so \(\wJ_u\geq J_u=t_u\) and \(\wD_u=\wJ_u+S_u\geq t_u+S_u\) for \(u=1, \ldots, c\). 
Suppose the inductive hypothesis holds for \(m=n\). Then we can apply the above monotonic formulae
and deduce that the inductive hypothesis holds for case \(m=n+1\) too. Thus the theorem follows by mathematical induction.

We now prove the claimed Kiefer-Wolfowitz domination between cases 3 and 4 for all times $t\geq T'$. (The proof that case 2 dominates case 3 for $t\geq T$ follows similarly.) To do this we construct two new  $\cdot/\cdot/c\;[FCFS]$ processes (called cases \(3'\) and \(4'\)) as follows: both systems have arrival times $t_i' = \max\{t_i,T'\}$, and case \(3'\) (respectively \(4'\)) has service durations given by the residual service durations in case 1 (respectively case 4) at time $T'$. That is, the system in case \(3'\) has service durations $R^{(3)}_1, R^{(3)}_2,\dots$ where
\[ R^{(3)}_i = \max\{J^{(1)}_i+S_i,T'\} - \max\{J^{(1)}_i,T'\} \,, \]
case \(4'\) has service durations $R^{(4)}_1, R^{(4)}_2,\dots$ where
\[ R^{(4)}_i = \max\{J^{(4)}_i+S_i,T'\} - \max\{J^{(4)}_i,T'\} \,, \]
and where $J^{(1)}_i$ and $J^{(4)}_i$ are the times of $i^\text{th}$ initiation of service in cases 1 and 4 respectively (see Figure~\ref{fig:process3'}).

\begin{figure}[t]
\centering
\includegraphics[width=\textwidth]{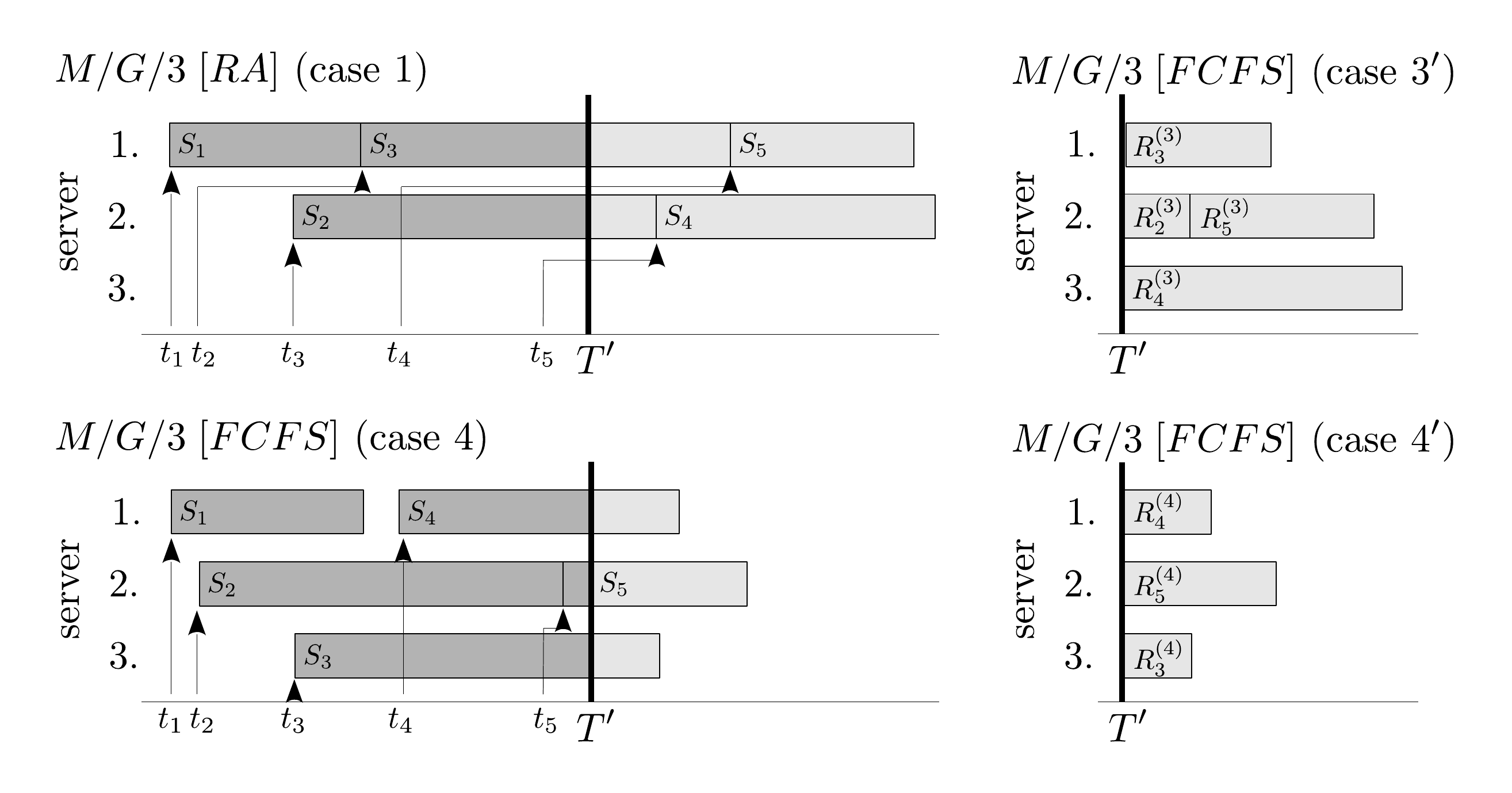}
\caption{Illustration (with three servers) of how cases $3'$ and $4'$ are instantiated at time $T'$ from the evolution over $[0,T')$ of cases 1 and 4 respectively. Top-left diagram shows arrivals at times $t_1,\dots,t_5$ and a possible allocation to servers 1--3 in the $M/G/3\,[RA]$ process. Lengths of blocks represent size of workloads $S_1,\dots,S_5$; work completed by time $T'$ is shaded dark grey. Top-right diagram shows how the residual workloads are allocated to servers at time $T'$ in case $3'$ (under FCFS). The bottom row shows how the same set of arrivals and service durations are handled by the FCFS process in case 4, and how the corresponding residual workloads are used to instantiate case $4'$ at $T'$.}\label{fig:process3'}
\end{figure}

Now consider the Kiefer-Wolfowitz vectors for these processes. We claim that (with subscripts corresponding in an obvious way to the cases being considered)
\begin{enumerate}
\item[(i)] $\mathbf{V}_{4'}(T') = \mathbf{V}_{4}(T')$
\item[(ii)] $\mathbf{V}_{3'}(T') = \mathbf{V}_{3}(T')$
\item[(iii)] $\mathbf{V}_{4'}(T') \tleq \mathbf{V}_{3'}(T')$, where $\tleq$ denotes coordinate-wise domination.
\end{enumerate}
The proof that $\mathbf{V}_{4}(t) \tleq \mathbf{V}_{3}(t)$ for all $t\geq T'$ follows from these three claims: it is immediate that the required domination holds at time $T'$; domination at all subsequent times follows by applying the final part of Theorem~\ref{thm:path-wise-domination1}, since cases 3 and 4 both operate a FCFS policy over $[T',\infty)$, with the same arrival times and associated service durations over this period. 

It therefore remains to prove claims (i)-(iii) above. For (i), note that any service duration $S_i$ which has completed service by time $T'$ under case 4 will correspond to an arrival time $t_i'=T'$ and residual service duration $R^{(4)}_i=0$ for case $4'$: such service durations therefore make no contribution to workload vectors $\mathbf{V}_{4'}(T')$ and $ \mathbf{V}_{4}(T')$. Next consider any pairs $(T',R^{(4)}_i)$ for which $R^{(4)}_i<S_i$ (i.e. customers who arrive by time $T'$ and who have had some, but not all, of their workload served by case 4 before this time). There can clearly be no more than $c$ such customers, and by construction they make the same positive contribution to both workload vectors $\mathbf{V}_{4'}(T')$ and $ \mathbf{V}_{4}(T')$ (subject to appropriate labelling of servers for case $4'$). 
The only other customers who contribute to  $\mathbf{V}_{4}(T')$ (and $\mathbf{V}_{4'}(T')$) are those who arrive before time $T'$ but who are yet to start service by this time (i.e. for which $J^{(4)}_i>T'$). These customers correspond to inputs of the form $(T',S_i)$ for case $4'$. Consider the first such arrival: in case 4 this customer is placed  in queue at its arrival time $t_i<T'$, being allocated to the server with the least residual workload at time $t_i$. But since all $c$ servers in case 4 must be busy  over the entire period $[t_i,T']$ (for if not, the customer arriving at time $t_i$ would necessarily have commenced service by $T'$), this server still has the least residual workload at time $T'$. It follows that this customer will be allocated to the same server in case $4'$ at $T'$. Arguing inductively along these lines, it is clear that all customers arriving over $[t_i,T']$ are allocated to identical servers
in cases 4 and $4'$, implying that $\mathbf{V}_{4'}(T') = \mathbf{V}_{4}(T')$, as required.

For (ii), note that the workload vector $\mathbf{V}_{3'}(T')$ is instantiated using the residual workloads at $T'$ from case 1 and then applying a FCFS policy, whereas the vector $\mathbf{V}_{3}(T')$ uses the residual workloads at $T'$ from case 3, again under FCFS (due to the change of service discipline in case 3 at time $T'$, as described in the statement of the theorem). But since the systems in cases 1 and 3 are identical over the period $[0,T')$, it is clear that $\mathbf{V}_{3'}(T') = \mathbf{V}_{3}(T')$.

Finally, claim (iii) follows from Theorem~\ref{thm:path-wise-domination1} applied to the two FCFS systems in cases $3'$ and $4'$, which use the same sequence of arrival times $(t_1',t_2',\dots)$ and possibly different sequences of service durations. But since $J^{(1)}_i\geq J^{(4)}_i$ for all $i$, it follows that $R^{(3)}_i \geq R^{(4)}_i$ (see Figure~\ref{fig:process3'} for an illustration), and so service durations for case $3'$ are at least as big as those for case $4'$, which provides the required monotonicity.
\end{proof}

We close this section with a standard lemma which assures us that actual numbers of customers in the systems
also obey the comparisons indicated in Theorems \ref{thm:path-wise-domination1} and \ref{thm:path-wise-domination2}, so long as the arrival processes agree. 
(Note that the same is \emph{not} true of total residual work-load.)
\begin{lem}\label{lem:domination3}
 Consider two queueing systems, such that arrivals happen at the same time for each system,
initiations of service happen earlier in the first than in the second (\(\wJ_m\geq\ J_m\) for all \(m\)),
and service durations are shorter in the first than in the second when indexed by order of initiation of service (\(\wS_m\geq S_m\) for all \(m\)).
Then numbers in the second system \(\wX_t\) exceed numbers in the first system \(X_t\) at any specific time \(t\).
\end{lem}
\begin{proof}
 Let \(X_t\), \(\wX_t\) be numbers in the system at time \(t\). Since \(\wJ_m+\wS_m \geq J_m+S_m\) for any \(m\), the representation \eqref{eq:departure0}
shows that departures happen later in the second system:
\[
 \wD_m \quad\geq\quad D_m \qquad \text{ for any }m.
\]
At a given time \(t\) we know that the same number \(\#\{m:t_m\leq t\}\) of customers have entered each of the systems. However the above inequality 
for departures shows that fewer have left the second system than the first. Accordingly
\[
 \wX_t\quad=\quad \#\{m:t_m\leq t\} - \#\{m:\wD_m\leq t\}\quad\geq\quad
 X_t\quad=\quad \#\{m:t_m\leq t\} - \#\{m:D_m\leq t\}\,.
\]
\end{proof}


\section[Simple dominated CFTP]{Simple dominated coupling from the past for $M/G/c$}\label{sec:domCFTP}
We seek a coupling from the past algorithm for an \(M/G/c\;[FCFS]\) queueing system.
The key step is to find a dominating process which is reversible. Sigman \cite{Sigman-2011} showed how to do this
if the system is \emph{super-stable} (arrival rate less than service rate of single server, i.e. $\rho<1$): observe
that the system is dominated by an \(M/G/1\;[FCFS]\) queue, notice that the workload process of the \(M/G/1\;[FCFS]\) queue
is the same as that of the same queue under processor sharing (\(M/G/1\;[PS]\)), exploit the dynamic time reversibility
of the \(M/G/1\;[PS]\) workload process (exchange residual workload for work so far completed on customers in service)
to simulate it backwards in time until empty, then 
use the observed departure times and associated service durations to evolve the Kiefer-Wolfowitz vector for the $M/G/c\;[FCFS]$ queue forwards until time zero.

This section shows how to improve on this by lifting the super-stability requirement, leaving only the minimal stability
requirement (arrival rate less than total service rates of all servers, i.e. $\rho<c$).
The idea is as follows: the results of the previous section show that the system is dominated by an \(M/G/c\;[RA]=[M_{\lambda/c}/G/1\;FCFS]^c\) queueing system with
random assignment allocation policy. Sigman \cite{Sigman-2011} noted that na{\"\i}ve pathwise domination fails; however we can and will exploit the
path-wise domination which holds when service durations are assigned in order of initiation of service. Again we can extend the 
dominating process backwards in time using the processor sharing representation. The simplest way to construct a dominated coupling from the past
is then to extend backwards in time till the dominating \(M/G/c\;[RA]\) system becomes completely empty,
because this allows us to identify service durations with initiations of service in a way which is consistent with further extensions backwards in time.
In effect we are exploiting the ``regenerative atom'' idea noted in \cite{KendallMoeller-2000}. The resulting sequences of arrival times
and service durations can then be used to construct a realization of an \(M/G/c\;[FCFS]\) queue that is subordinate to the dominating process.
Since this can be extended further back in time, using further emptying times of the dominating process, 
we have produced the tail end of a ``simulation from time minus infinity'' which must therefore be in equilibrium at time zero (for a more mathematical account of this idea, see 
 \cite{Kendall-2005a}).

So the steps of the algorithm are as follows:

\begin{algorithm}\label{alg:1}
The algorithm description involves some random processes and associated random quantities which are run backwards in time: 
such quantities are crowned with a hat (for example, \(\hat{Y}\) below). We summarize the algorithm in 4 steps.
\begin{enumerate}
 \item\label{item:PS}
 \emph{Consider a $[M/G/1\;PS]^c$ process $\hat{Y}$, run backwards
in time in statistical equilibrium. Make a draw from \(\hat{Y}(0)\), the state of the process at time zero.}
 \item\label{item:reverse-time}
 \emph{Simulate the \(c\) components of the reversed-time process $(\hat{Y}(\hat{t}):\hat{t}\geq0)$ over the range \([0,\hat{\tau}]\), where \(\hat{\tau}\)
 is the smallest reversed time such that all components are empty at \(\hat{\tau}\).}
 \item\label{item:RA}
 \emph{Use $(\hat{Y}(\hat{t}):\hat{t}\in[0,\hat{\tau}])$ to construct its (dynamic) time reversal, and thus to build 
$(Y(t):\tau\leq t\leq 0)$, an $M/G/c\;[RA]=[M/G/1\; FCFS]^c$ process (here we set \(\tau=-\hat{\tau}\)).}
 \item\label{item:FCFS}
 \emph{Use $Y$ to evolve $X$, an $M/G/c\; [FCFS]$ process, over $[\tau,0]=[-\hat{\tau},0]$, started in the empty state.}
\end{enumerate}
Because of the comparison theorems \ref{thm:path-wise-domination1} and \ref{thm:path-wise-domination2}, and Lemma \ref{lem:domination3},
we may further extend \(\hat{Y}\) forwards in reversed time, and thus \(Y\) backwards in time,
and use this construction to build further variants of \(X\) started in the empty state from any time earlier than \(-\tau\).
Suitably extended back in time, \(Y\) dominates all these versions; moreover agreement of any two variants \(X^{(1)}\) and \(X^{(2)}\)
is enforced at the point when \(Y\) visits the empty state subsequent to both of their starting times. The arguments discussed in
\cite{Kendall-2005a} then show that the common value \(X(0)\) of all these variants must be a draw from the statistical equilibrium of
the \(M/G/c\;[FCFS]\) queue under consideration.
\end{algorithm}
We now discuss in turn the details of each of these steps.

\subsubsection*{Step \ref{item:PS}: generating a draw from the processor sharing queue system in equilibrium.}
Thanks to the Pollaczek-Khintchine formula for an \(M/G/1\) queue, we know that the equilibrium distribution for the residual workload of $\hat{Y}_j$ at time \(0\)
(where $j\in\{1,\dots,c\}$ denotes the $j^{th}$ server) is distributed as 
\[ 
\sum_{i=1}^{\hat{Q}_j(0)} \hat{S}_{j,i}^e(0) \,.
\]
Here $\hat{S}_{j,i}^e(0)$ are independent and identically distributed draws from the distribution of service durations in equilibrium,
with distribution function
\[ 
G_e(x) \quad=\quad \mu \int_0^x \bar G(y) \operatorname{d} y \qquad\text{ for } x\geq0 
\]
(for \(\bar G(y)=1-G(y)\) the complementary distribution function of a service duration),
while
$\hat{Q}_j(0)$ is an independent random variable with geometric distribution given by 
\[
\Prob{\hat{Q}_j(0)=n}\quad=\quad(\rho/c)^n(1-\rho/c) \qquad \text{ for }n\geq 0\,.
\]
However, we need to know the \emph{total} (not only residual) workload brought by each of the customers currently in service. 
Arguing as in \cite{Sigman-2011}, or using a dynamic reversibility argument, 
we do this by simulating from the stationary \emph{spread} distribution for each of the $\hat{Q}_j(0)$ customers being served by server $j$ at time 0, giving draws $H_{j,1},H_{j,2},\dots,H_{j,\hat{Q}_j(0)}$: these represent the total workload brought by each customer. 
Here the stationary spread distribution of service durations is the length-biased variant of \(G\), 
with complementary distribution function given by
\[ 
\bar G_s(x) \quad=\quad \mu x \bar G(x) + \bar G_e(x) \qquad\text{ for } x\geq0 \,.
\]
(Our assumption that $\Expect{S^2}<\infty$ guarantees that the spread distribution has finite mean.) We then draw independent $\textrm{Uniform}[0,1]$ random variables $U_{j,i}$ and set $\hat{S}_{j,i}^e(0) = U_{j,i}H_{j,i}$ to 
represent the residual workloads at time \(0\).

Finally, since all of the servers in $\hat{Y}$ work independently of each other, 
$c$ independent draws from this distribution deliver an equilibrium draw from $\hat{Y}$. Set 
\[
\hat{Y}_j(0) \quad=\quad R(S^e_{j,1}(0),\dots,S^e_{j,\hat{Q}_j(0)}(0)) 
\]
to be this draw from equilibrium, viewed as a list of workloads $S^e_{j,i}(0)$ listed in increasing order (\(R\) being the re-ordering operator mentioned in Section \ref{sec:dynamic}).

\subsubsection*{Step \ref{item:reverse-time}: evolving the processor sharing queue system in reverse time till it empties.}
We record the \([M/G/1\;PS]^c\) queueing system as follows: at (reversed) time \(\hat{t}\), the system is defined by
\begin{align*}
 \hat{Q}_j(\hat{t})\quad&=\quad \text{the number of customers for server \(j\) at time } \hat t\,,\\
 S^e_{j,i}(\hat{t})\quad&=\quad \text{the residual workload of customer \(i\)  for server \(j\) at time } \hat{t}\,,\\
 \hat{Y}_j(\hat{t})\quad&=\quad R(S^e_{j,1}(\hat{t}),\dots,S^e_{j,\hat{Q}_j(\hat{t})}(\hat{t}))\,,\\
 \hat{Y}(\hat{t})\quad&=\quad (\hat{Y}_1(\hat{t}), \hat{Y}_2(\hat{t}), \ldots,\hat{Y}_c(\hat{t}))\,. 
\end{align*}
It is convenient to write \(|\hat{Y}_j(\hat{t})|=\hat{Q}_j(\hat{t})\), 
and \(|\hat{Y}(\hat{t})|=|\hat{Y}_1(\hat{t})|+|\hat{Y}_2(\hat{t})|+\ldots+|\hat{Y}_c(\hat{t})|\).

If $|\hat{Y}(0)| = 0$ (so there is no residual workload left in the system at all) then set $\hat{\tau}=0$ and stop simulating.

Otherwise, use event-based simulation.
Calculate the next event time after \(\hat{t}\) as follows: 
For each \(j\in\{1,2,\ldots,c\}\), all of the $\hat{Q}_j$ customers of server \(j\) are served simultaneously by server $j$ at rate $1/|\hat{Y}_j(t)|$ 
until either one of the customers of one of the servers has been completely served (and then leaves the system)
or a new customer arrives (at rate $\lambda$) to be served by one or another of the servers. Reset $\hat{t}$ accordingly.
\begin{itemize}
\item If the event is an \emph{arrival}, then generate a new service duration $S$ for the customer (using distribution $G$) and choose a server $j\in\{1,\dots,c\}$ 
to which the customer is allocated. The customer is placed in service, so increment $\hat{Q}_j$ by \(+1\) (so that the per-customer service rate of server $j$
drops accordingly).
\item If the event is a \emph{departure}, then record the \emph{departure time} and the \emph{full} service duration of the departing customer; 
increment $\hat{Q}_j$ by \(-1\) (so that the per-customer service rate of server $j$
increases accordingly). If $|\hat{Y}(\hat{t})| = 0$ then set $\hat{\tau}=\hat{t}$ and stop simulating.
\end{itemize}
If $\hat{\tau}>0$, then record the departure times of customers as $0\leq \hat{t}_1\leq \hat{t}_2 \leq \dots \leq \hat{t}_k=\hat{\tau}$, and 
record the associated (full) service durations as $S_1,\dots,S_k$.

\subsubsection*{Step \ref{item:RA}: dynamic time-reversal and construction of the $M/G/c\;[RA]=[M/G/1\; FCFS]^c$ 
dominating process.}
Let the $M/G/c\;[RA]=[M/G/1\; FCFS]^c$ system $Y$ start from the empty state at time \(\tau=-\hat{\tau}\) and run forward in time. We let \(|Y(t)|\) denote the total number of customers in \(Y\) at time \(t\).
Arrivals occur at times $\tau = -\hat{t}_k \leq -\hat{t}_{k-1} \leq\dots \leq -\hat{t}_1$. 
The customer arriving at time $-\hat{t}_i$ has associated service duration $S_i$ (obtained from records kept as specified in Step \ref{item:reverse-time} above),
and is allocated to the same server that completed service $S_i$ in $\hat{Y}$. 
Reorder the set of service durations according to the corresponding \emph{initiation of service durations} in the forwards queueing system $Y$. 
Denote this ordered list by $\mathcal S'= (S'_1,\dots,S'_k)$: if $J^Y_i$ is the time of initiation of service $S'_i$ in $Y$, 
then $\tau = J^Y_1 \leq  J^Y_2 \leq \dots \leq J^Y_k$.

Note that it is possible for $J^Y_i$ to be positive, in the case where $Y$ has customers in the queue at the (terminal!) time \(0\) 
who have yet to commence service. 
In the event that $J^Y_k>0$, we extend the simulation of $Y$ further into the future by drawing extra (independent) arrival times over the period $(0,J^Y_k]$, along with associated service durations. 
This results in a set of additional service durations with associated times of initiation of service: 
these extra services are then added to the list $\mathcal S'$ in order of these times of initiation of service. 
(Note that this implies a potential change in index for service durations $S'_i$ for which $J^Y_i>0$.)

This gives us a method of constructing a stationary version of $Y$ started arbitrarily far back into the past and run until the time when all customers in the system at time zero have commenced service: 
for example, our simulation of $\hat{Y}$ can be extended to the second time $\hat \tau'>\hat\tau$ of emptying, 
and then these additional departure times and service durations used to feed $Y$ over the corresponding period of forward time $[\tau'.\tau)$. 
Since the workload of the $M/G/1$ queue is invariant under changes of work-conserving discipline, 
it follows that if $Y$ starts from empty at $\tau'$ then it will again be empty just before the arrival at time $\tau$.

\subsubsection*{Step \ref{item:FCFS}: construction of the target process $M/G/c\; [FCFS]$ process.}
Start the $M/G/c\; [FCFS]$ queue $X$ from the empty state at time \(\tau\), 
and let it evolve (using \eqref{eqn:KW-recursion}) by generating arrivals at times 
$\tau = -\hat{t}_k \leq -\hat{t}_{k-1} \leq\dots \leq -\hat{t}_1$ (i.e. the same arrival times as used for $Y$), 
but with service duration $S'_i$ now allocated to the arrival $-\hat{t}_i$. 
Since customers are served by $X$ in order of arrival, this means that service durations are once again allocated by time of initiation of service, 
i.e. $J^X_1 \leq J^X_2 \leq \dots$. 
The domination arguments of Section~\ref{sec:domination} permit us to argue that $J^X_i \leq J^Y_i$ for $i=1,\dots,k$, 
and so $X$ satisfies $|X(t)| \leq |Y(t)|$ for all $t\in[\tau,0]$, where \(|X(t)|\) denotes the total number of customers in \(X\) at time \(t\). 
(Note, however, that it is certainly \emph{not} the case that the residual workload in $X(t)$ is necessarily dominated by that in $Y(t)$.) 
Return $X(0)$ as a draw from equilibrium of an $M/G/c\; [FCFS]$ process.


\section[Sandwiching]{Sandwiching for Dominated CFTP algorithm}\label{sec:sandwich}
The algorithm described in Section \ref{sec:domCFTP} is inefficient, because it uses the regenerative atom which is the empty system state.
For typical applications of \(M_\lambda/G/c\) queueing systems, we would expect
\(1\ll \rho <c\),
so that the system would frequently visit states where no more than \(c\) people were in the system, but would only rarely visit the
empty state.

A more efficient dominated coupling from the past algorithm exploits the domination results (Theorems \ref{thm:path-wise-domination1}, \ref{thm:path-wise-domination2} and Lemma \ref{lem:domination3}) to establish \emph{sandwiching}. The idea is to stop the backward-in-time simulation of 
the $[M/G/1\;PS]^c$ process $\hat{Y}$ at some  time \(\hat{T}\) well short of the time required to achieve empty state, but then to construct a lower envelope
\(M/G/c\; [FCFS]\) process \(L\) (started at the empty state at time \(-\hat{T}\)) and an upper envelope \(M/G/c\; [FCFS]\) process \(U\) (started using the state of the forwards dominating $M/G/c\;[RA]=[M/G/1\; FCFS]^c$ process \(Y\)
at time \(-\hat{T}\)), and to evolve these using the arrival times and service durations derived from \(Y\) in such a way that 
(a) at any given time, the number of people in \(L\) lies below the number in \(U\) which in turn lies below the number of people in \(Y\),
(b) similar envelope processes begun at earlier times sandwich themselves between \(L\) and \(U\) 
(the so-called ``sandwiching property''), in the sense of coordinate-wise domination of Kiefer-Wolfowitz workload vectors.
It follows from the theory of dominated coupling from the past  \cite{KendallMoeller-2000,Kendall-2005a} that if we then successively decrease \(-\hat{T}\) till eventually \(L(0)=U(0)\) then the common state of \(L(0)=U(0)\) will be a draw from the equilibrium (this depends crucially on the sandwiching property mentioned above, which must not be neglected in implementation).
The delicate issue in all this is exactly the requirement to maintain sandwiching. This requires us to match service durations to times of initiation of service, not just with respect to individual pairs of envelope processes, but also as between a couple of pairs begun at different times.
The trick is to extend the simulation of \(\hat{Y}\) beyond \(\hat{T}\) so that matching may be carried out in a stable way.

\begin{algorithm}\label{alg:2}
\begin{enumerate}
 \item\label{item:PS2}
 \emph{Consider a $[M/G/1\;PS]^c$ process $\hat{Y}$, run backwards
in time in statistical equilibrium. Make a draw from \(\hat{Y}(0)\), the state of the process at time zero.}
\item\label{item:reverse-time2} \emph{Fix a suitable positive \(\hat T = -T\). 
Evolve the queue for server $j$ of $\hat{Y}$ (independently of all other servers) until the first time $\hat\tau_j\geq \hat T$ that \emph{this server} is empty, for $j=1,\dots,c$.}
\item\label{item:RA2} \emph{Construct $Y_j$, an $M/G/1\; [FCFS]$ process over the corresponding reversed time interval $[-\hat\tau_j,0]$, for $j=1,\dots,c$.}
\item\label{item:FCFS2} \emph{Produce lists of service durations and arrival times, $\mathcal L^*_{T}$ and $\mathcal L_T$.}
\item\label{item:upper2} \emph{Construct an upper sandwiching process, $U_{[T,0]}$ over $[T,0]$.}
\item\label{item:lower2} \emph{Construct a lower sandwiching process, $L_{[T,0]}$ over $[T,0]$.}
\item\label{item:coalescence2} \emph{Check for coalescence.}
\end{enumerate}
\end{algorithm}
We now discuss in turn the details of each of these steps.

\subsubsection*{Step \ref{item:PS2}: Produce a sample from the stationary distribution of the $[M/G/1\; PS]^c$ process $\hat{Y}$.}
This is performed exactly as in Algorithm \ref{alg:1}.

\subsubsection*{Step \ref{item:reverse-time2}: Evolve the queue for \emph{each server} of $\hat{Y}$ independently until empty.}

Record departure times and associated (full) service durations for each server; simulate the queue served by server $j$ (as in Step~\ref{item:reverse-time} of Algorithm \ref{alg:1}) until time 
\[
\hat{\tau}_j = \inf\{\hat{t}\geq \hat{T}:|\hat{Y}_j(\hat{t})| = 0\}, \quad j=1,\dots,c. 
\]

\subsubsection*{Step \ref{item:RA2}: Construct $Y_j$, an $M/G/1\;[FCFS]$ process over the corresponding reversed time interval, for $j=1,\dots,c$.}

For each server \(j\in\{1,2,\ldots,c\}\), we simulate $Y_j$ starting in the empty state at time \(\tau_j=-\hat \tau_j\), 
and we feed the simulation with arrival times and associated service durations corresponding to the recorded departures from $\hat{Y}_j$. If \(t^+>0\) is the \emph{positive} time by which all customers in \(Y\) at time zero have initiated service, 
then extend each \(Y_j\) simulation to cover the range \((\tau_j,t^+]\), as in the detail of Step \ref{item:RA} of Algorithm \ref{alg:1}. 
Note that, since the $M/G/1\;[FCFS]$ process $Y_j$ starts from empty at time $\tau_j$, the path of $Y_j$ over \([0,t^+]\) will remain unchanged if we decrease the value of $T<0$.
Furthermore, since $\tau_j \leq T$ for $j=1,\dots,c$, we have in fact established the path of $Y$, an $[M/G/1\; FCFS]^c$ process, over the interval $[T,0]$.

\subsubsection*{Step \ref{item:FCFS2}: Produce lists of service and arrival times, $\mathcal L^*_T$ and $\mathcal L_T$.}

Form the union of all arrival times observed in each $Y_j$ over the interval $[\tau_j,0]$, and order them as $t_1\leq t_2\leq \dots\leq t_n$. Similarly, form the union of all pairs $(J,S)$ of time $J$ of initiation of service and associated service duration $S$ from each $Y_j$ over the interval $[\tau_j,t^+]$, and order these in increasing order of $J$. Let 
$\mathcal L_T^* = \{ (t_k,J_k,S_k): t_k\leq T \}$ and $\mathcal L_T= \{ (t_k,J_k,S_k): t_k> T  \}$.
Finally, for each $(t_k,J_k,S_k)\in\mathcal L_T^*$, replace the arrival time $t_k$ by $T$, $J_k$ by $J_k\vee T$, and the service duration $S_k$ by its \emph{residual workload} at time $T$. That is, replace $(t_k,J_k,S_k)$ by \((T,J_k\vee T,R_k)\), where \(R_k=(J_k+S_k)\vee T-(J_k\vee T)\).

\subsubsection*{Step \ref{item:upper2}: Construct an upper sandwiching process, $U_{[T,0]}$ over $[T,0]$.}

We construct an $M/G/c\;[FCFS]$ process $U_{[T,0]}$ over $[T,0]$ by starting from empty at time $T$ and feeding it the arrival times and service durations read first from $\mathcal L_T^*$ and then from $\mathcal L_T$. The intention here is that $U$ can be seen to be a process which switches from the $M/G/1\;[RA]$ queue $Y$ to an $M/G/c\;[FCFS]$ queue at time $T$: Theorem~\ref{thm:path-wise-domination2} guarantees that $|U|$ (the number of customers in the upper process) will be dominated by $|Y|$.

\subsubsection*{Step \ref{item:lower2}: Construct a lower sandwiching process, $L_{[T,0]}$ over $[T,0]$.}

In a similar manner we construct an $M/G/c\;[FCFS]$ process $L_{[T,0]}$ over $[T,0]$ by starting from empty at time $T$ and feeding it the arrival times and service durations read once again from $\mathcal L_T^*$ and then from $\mathcal L_T$, but now \emph{with all the service durations in $\mathcal L_T^*$ set to zero}. Theorems~\ref{thm:path-wise-domination1} and \ref{thm:path-wise-domination2} ensure that $L_{[T,0]}\tleq U_{[T,0]}$ (here $\tleq$ denotes coordinate-wise domination of the Kiefer-Wolfowitz workload vectors).

\subsubsection*{Step \ref{item:coalescence2}: Check for coalescence.}

If the \emph{residual workload vectors} of $U_{[T,0]}(0)$ and $L_{[T,0]}(0)$ agree, return their common value as a draw from equilibrium of our target process $X$. If not, then replace $\hat T$ by $\hat T'>\hat T$, and return to Step \ref{item:reverse-time2}: extend the paths of $\hat{Y}_j$ until they have each emptied at some time $\hat{\tau_j}' \geq \hat T'$ etc., and continue as before. 

\medskip

The reader may be concerned that coalescence here occurs when the residual workload vectors first coincide, apparently without requiring equality of numbers of customers in system.
However, under FCFS, a disparity of numbers together with equality of residual workload vectors would require at least one of the two systems to have strictly more than \(c\) customers in system.
As already remarked after the proof of Theorem \ref{thm:path-wise-domination1}, 
coalescence for Algorithm \ref{alg:2} can only occur when both processes have idle servers, and in this case equality of residual workload vectors implies equality of the numbers of customers in system.

Since $L_{[T,0]}$ is a version of our target $M/G/c\;[FCFS]$ process started from empty, a standard dominated CFTP argument  \cite{KendallMoeller-2000,Kendall-2005a} shows that the above algorithm really does return a perfect draw from the correct equilibrium distribution, as long as the upper and lower processes really do satisfy the ``sandwiching property''. The following theorem establishes a rigorous validation of sandwiching for Algorithm \ref{alg:2}.
\begin{thm}\label{thm:sandwiching}
Let $L_{[T,0]}$ and $L_{[T',0]}$ (respectively $U_{[T,0]}$ and $U_{[T',0]}$) be lower (respectively upper) sandwiching processes, defined as above, started at times $T'<T<0$. Then for all times $t\in[T,0]$,
\[ L_{[T,0]}(t) \quad \tleq\quad L_{[T',0]}(t)  \quad\tleq\quad U_{[T',0]}(t) \quad\tleq\quad U_{[T,0]}(t) \,, \]
where $\tleq$ once again denotes coordinate-wise domination of the Kiefer-Wolfowitz workload vectors.
\end{thm}

\begin{proof}
Let $\hat\tau_j$ (respectively $\hat\tau_j'$) be the first time after $\hat T=-T$ (respectively $\hat T'=-T'$) that $\hat Y_j$ empties, for $j=1,\dots, c$. As noted at the end of Step~\ref{item:RA2}, when we extend the simulation of $\hat Y_j$ from $[0,\hat\tau_j]$ to $[0,\hat\tau_j']$ its path over $[0,\hat\tau_j]$ is unchanged. It follows that the list $\mathcal L_T$ (created in Step~\ref{item:FCFS2} above) is unchanged by such an extension, i.e.
\[ \mathcal L_T = \{ (t,J,S)\in \mathcal L_{T'}: t> T \}. \]
Furthermore, any additional entries created in $\mathcal L_T^*$ when extending from $\hat\tau_j$ to $\hat \tau_j'$ (which must satisfy $t_k\leq T$) have zero residual service durations at time $T$. Extending the simulation of $\hat Y_j$ ($j=1,\dots,c$) from $[0,\hat\tau_j]$ to $[0,\hat\tau_j']$ thus has no effect on the paths of $U_{[T,0]}$ and $L_{[T,0]}$.
We may therefore assume that the lists $\mathcal L_T^*$, $\mathcal L_T$, $\mathcal L_{T'}^*$ and $\mathcal L_{T'}$ are \emph{all} constructed by running each $\hat Y_j$ process over the longer intervals $[0,\hat\tau_j']$.

Now we simply observe that $U_{[T',0]}$ is (as remarked in Step~\ref{item:upper2} above) a process which switches from the $M/G/1\;[RA]$ queue $Y$ to an $M/G/c\;[FCFS]$ queue at time $T'$, whereas $U_{[T,0]}$ switches from $Y$ to $M/G/c\;[FCFS]$ at a later time $T>T'$: Theorem~\ref{thm:path-wise-domination2} shows that $U_{[T,0]}$ must therefore dominate $U_{[T',0]}$. Similarly, $L_{[T,0]}$ and $L_{[T',0]}$ are two $M/G/c\;[FCFS]$ processes, which can be viewed as both starting from empty at time $T'$ and with all service durations corresponding to pre-$T$ arrival times set to zero for $L_{[T,0]}$: Theorem~\ref{thm:path-wise-domination1} shows that $L_{[T',0]}$ dominates $L_{[T,0]}$, as required.
\end{proof}

Bearing in mind the remark made at the end of Theorem \ref{thm:path-wise-domination1}, 
we see that coalescence of the two sandwiching processes can occur only when both upper and lower sandwiching processes have strictly fewer than \(c\) individuals in system.
There follows an almost obvious remark: in case \(c=1\) Algorithm \ref{alg:2} offers no advantage over Algorithm \ref{alg:1}, which itself reduces to the \(c=1\) case of  \cite{Sigman-2011}.
However if \(c>1\) then it is possible for Algorithm \ref{alg:2} to produce coalescence when started prior to the latest time (prior to time \(0\)) at which the equilibrium queue 
has an idle server. For large \(c\) it follows that Algorithm \ref{alg:2} offers substantial practical advantages in terms of reduced run-time.


\section[Assessment of algorithms]{Assessment of algorithms for $M/M/c$ case \emph{etc}}\label{sec:mmc}

So far we have introduced, and proved the correctness of, two algorithms for perfectly sampling from the stationary distribution of the Kiefer-Wolfowitz workload vector for stable $M/G/c$ queues. 
In this section we briefly investigate and compare the performance of these algorithms, mainly in the special case when service durations are Exponentially distributed (i.e. for an $M/M/c$ queue). 
We begin with a discussion of choice of back-off strategy for Algorithm \ref{alg:2}, 
and then present some simulation results which indicate that this algorithm may be substantially faster than (the rather na{\"\i}ve) Algorithm \ref{alg:1}. 
These observations are reinforced by theoretical bounds on the run-time of the two algorithms, which can be found in Section \ref{sec:convergence-rates}. 
We do not here present
a complete analysis of our algorithms' performance, but we do elucidate
the relative efficiency of Algorithm \ref{alg:2}.

\subsection{Back-off strategies}\label{sec:back-off}

In Algorithm \ref{alg:2} it is necessary to specify a method for choosing the sequence of times $\{\hat T, \hat T', \dots \}$ at which to check for coalescence. 
We briefly discuss two options.
The first is to use the well-known `binary back-off' method (set $\hat T=1, \hat T' = 2$, and continue to double in this way for as long as necessary), as is employed in many CFTP algorithms. The second is to use a sequence of stopping times determined by the dominating $\hat Y_j$ processes. For $j=1,\dots,c$ let $\hat \tau_j = \inf\{t>0:|\hat Y_j| = 0\}$, and let 
$\hat \tau^-$ and $\hat\tau^+$ be the minimum and maximum of these times; suppose that server $j^-$ is the one that empties for the first time at $\hat\tau^-$. The first time at which we can possibly check for coalescence is when running $U$ and $L$ over $[-\hat\tau^-,0]$. If this doesn't lead to coalescence then the path of $\hat Y_{j^-}$ needs to be extended until it empties once again, at which point we update the values of $\hat \tau^-$, $\hat\tau^+$ and $j^-$ and repeat. However, since server $j^-$ is starting from empty at time $\hat\tau^-$, it is quite likely to empty again after only a relatively short period of time, and it may therefore be computationally expensive to check for coalescence as soon as this server is once again empty. In what follows we make use of a binary back-off strategy whenever making use of Algorithm \ref{alg:2}.


\subsection[Example of simulation output]{Example of simulation output}\label{sec:simulation}

Both of our two algorithms produce a perfect sample from the stationary distribution of the Kiefer-Wolfowitz workload vector. Figure \ref{fig:KWuniform} shows the result of using Algorithm \ref{alg:2} for an $M/G/c$ queue with $\lambda=c=25$ and service distributions following a $\text{Uniform}[0,1]$ distribution; here we have chosen to display the last six coordinates of the workload vector (for which, recall, the coordinates are ordered monotonically by remaining workload).

\medskip
\begin{figure}[th]
\centering
\includegraphics[width=\textwidth]{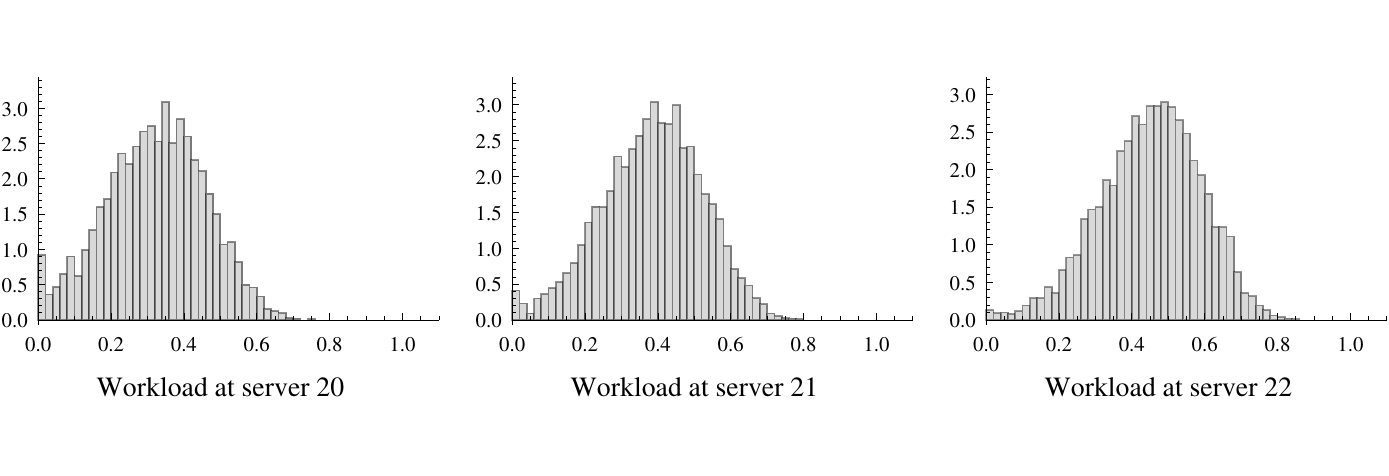}\vspace{-0.4cm}
\includegraphics[width=\textwidth]{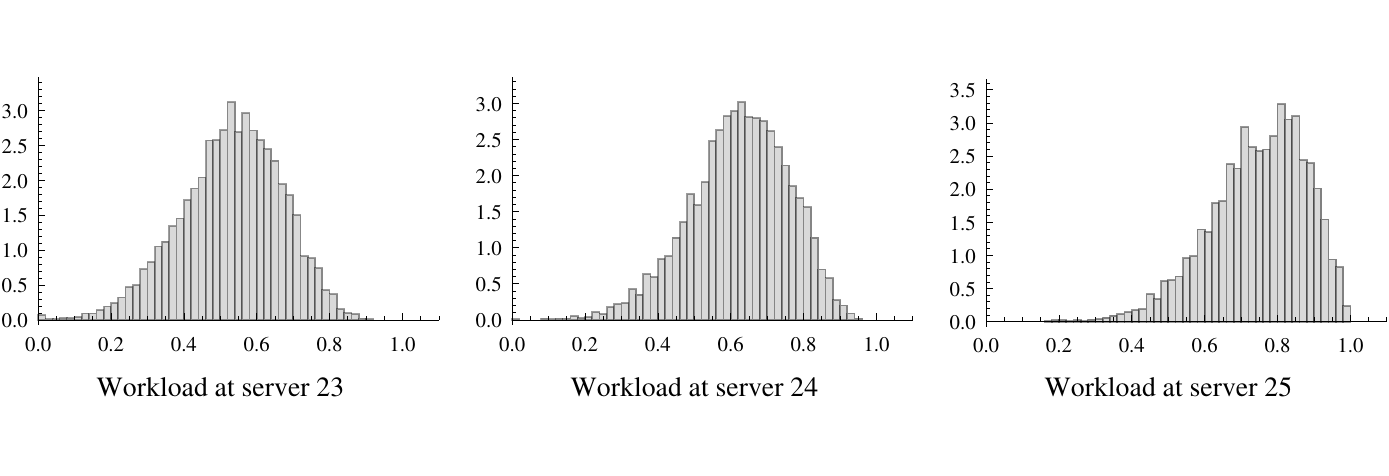}\vspace{-0.6cm}
\caption{Equilibrium distribution of the final 6 coordinates 
of the Kiefer-Wolfowitz workload vector when $\lambda=c=25$ and service durations are uniformly distributed on $[0,1]$. (Produced from 5,000 draws using Algorithm \ref{alg:2}.)   
}\label{fig:KWuniform}
\end{figure}

\begin{figure}[ht]
\centering
\includegraphics[width=0.6\textwidth]{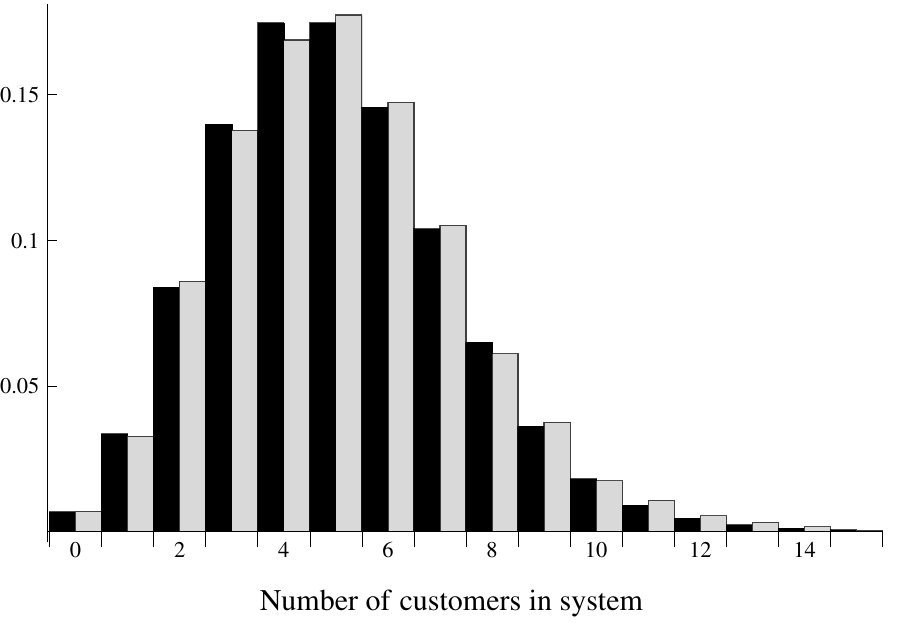}
\caption{Number of customers for an $M/M/c$ queue in equilibrium when $\lambda=10$, $\mu = 2$ and $c=10$. 
Black bars show the theoretical number of customers in the system; light grey bars show the result of 5,000 draws using Algorithm \ref{alg:2}.}\label{fig:hist1}
\end{figure}


\begin{figure}[h]
\centering
  \includegraphics[width=0.6\textwidth]{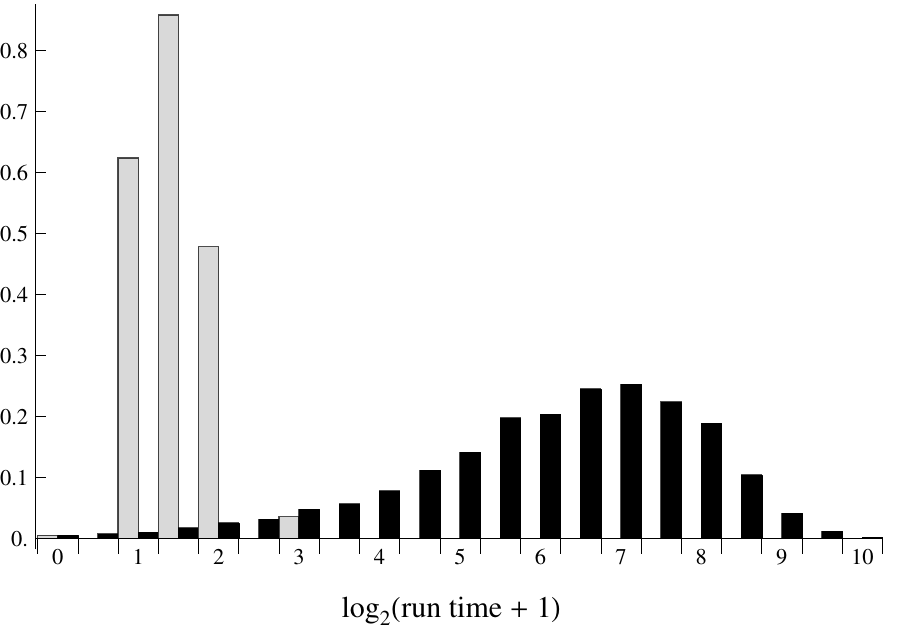}
  \caption{Distribution of time taken for coalescence to be detected under Algorithms \ref{alg:1} and \ref{alg:2} applied to an \(M/M/c\) queue, 
for 5,000 runs with $\lambda=10$, $\mu=2$ and $c=10$. 
Black bars show the distribution of $\log_2(\hat\tau+1)$ for Algorithm \ref{alg:1}, where $\hat\tau$ is the first time at which $\hat Y$ empties; light grey bars show the distribution of $\log_2(\hat T +1)$ for Algorithm \ref{alg:2}, where $\hat T$ is the smallest time needed to detect coalescence using binary back-off.
}\label{fig:hist2}
\end{figure}

When service durations follow an Exponential distribution (\emph{i.e.} $X$ is an $M/M/c$ queue) there is a well-known closed form for the distribution of the number of customers in the system under stationarity:
\[ 
\pi_k
\quad=\quad \left(\frac{\rho}{c}\right)^k \frac{c^{(k\wedge c)}}{(k\wedge c)!} \; \pi_0 \,, \quad\text{ for } k \geq 0\,. 
\]
We have compared the theoretical distribution to the empirical distribution obtained by output from large numbers of runs of Algorithm \ref{alg:2} for a wide variety of different sets of parameter values and achieved good agreement: by way of illustration, the result of doing this when $\lambda=10$, $\mu=2$ and $c=10$ is shown in Figure~\ref{fig:hist1}. 
Note that these parameters clearly satisfy $1< \rho =\lambda/\mu < c$, and so this is an example of a stable, but not super-stable, queue.
A chi-squared test between the theoretical and observed distributions here gave a \(p\)-value of \(0.62\), indicating good agreement.

It is also of interest to compare how far one has to simulate the
dominating process $\hat Y$ for each algorithm, and we have performed
such a comparison for a variety of sets of parameter values. In
Figure~\ref{fig:hist2} we give an indication of how much quicker it
may be to detect coalescence via Algorithm \ref{alg:2} rather than
simply waiting for $\hat Y$ to empty (as in Algorithm
\ref{alg:1}). For this example we once again set $\lambda=10$, $\mu=2$
and $c=10$, and we performed 5,000 runs of each algorithm. For
Algorithm \ref{alg:1} we recorded the value of $\hat\tau$  (the time
taken for $\hat Y$ to empty), while for Algorithm \ref{alg:2} we
employed a binary back-off approach (as is common in many CFTP
algorithms) and recorded the minimum value of $\hat T$ needed to
determine coalescence of our upper and lower sandwiching
processes. Note that the binary back-off approach means that it is
possible for Algorithm \ref{alg:2} to take longer than Algorithm
\ref{alg:1} to detect coalescence (e.g. if $\hat Y$ empties at time
$0<\hat\tau<1$ then Algorithm \ref{alg:2} won't detect this until
$\hat T = 1$; similar phenomena arise in several perfect simulation
algorithms involving binary back-off) but that in general Algorithm
\ref{alg:2} is significantly faster. In Figure~\ref{fig:hist3}(a) and
(b) we show similar run-time results for Algorithm \ref{alg:2} using substantially larger values of $\lambda$ and $c$ (while maintaining $\rho=\lambda/2$). The coalescence time \(\hat\tau\) here clearly does not increase significantly: in the following section we give an heuristic argument which explains why this is to be expected, at least when service times are exponentially distributed.

Of course, such a comparison does not take into account the additional computational demands of checking for coalescence (usually repeatedly) in Algorithm \ref{alg:2}, nor the fact that some of the servers in the $[M/G/1]^c$ process may not empty until a time which is significantly greater than $\hat T$ (especially when $\rho$ is close to 1), and so this is by no means a complete discussion of the relative efficiency of each algorithm. 
However, it emphasizes just how much sooner it is possible for coalescence to be detected, without the need to wait for the dominating process to empty completely.
(Note that the computational demands of Algorithms \ref{alg:1} and \ref{alg:2} may be compared as follows: running from an emptying time \(\tau<0\), Algorithm \ref{alg:1} requires simulation of an \([M/G/1]^c\) and an \(M/G/c\;[FCFS]\).
Running from a time \(T<0\), Algorithm \ref{alg:2} requires simulation of an \([M/G/1]^c\) and \emph{two} coupled \(M/G/c\;[FCFS]\). Bearing this in mind, 
and exploiting the remark after the proof of Theorem \ref{thm:path-wise-domination1},
the choice between Algorithms \ref{alg:1} and \ref{alg:2} should depend
on heuristic comparison of first moments of emptying time \(\tau\) and the latest time (prior to \(0\)) at which the equilibrium queue \(M/G/c\;[FCFS]\) has an idle server.)


\subsection{Notes on convergence rates}\label{sec:convergence-rates}

\begin{figure}[h]
\centering
  \includegraphics[width=0.48\linewidth]{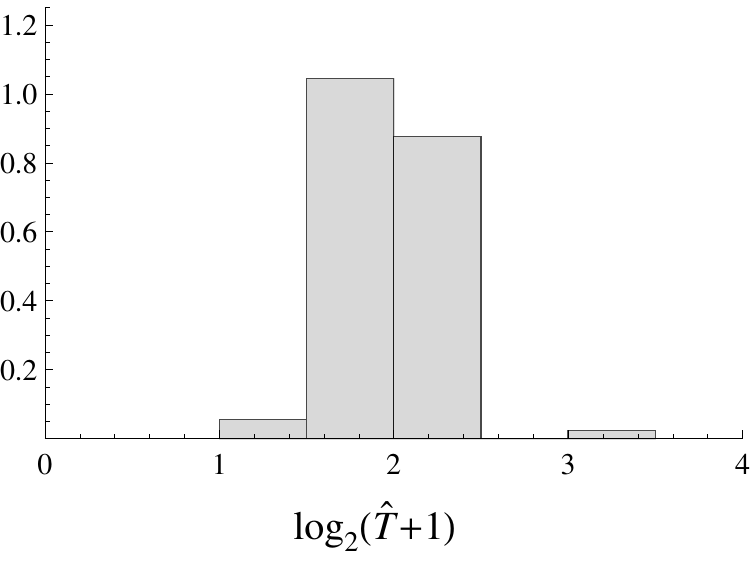}
  \includegraphics[width=0.48\linewidth]{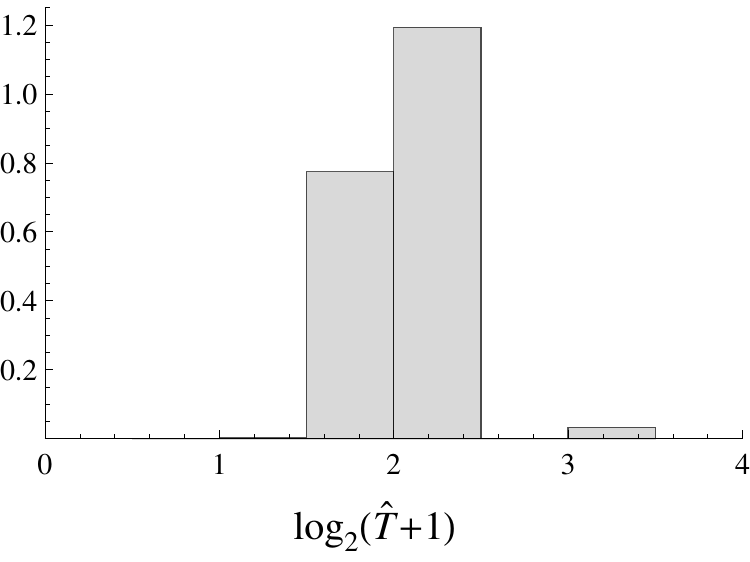}
\hspace{5em}(a) \hspace{0.44\linewidth} (b)
\caption{Sample of run-time distributions for a selection of $M/M/c$
  queues. (a) Coalescence time (measured as in Figure~\ref{fig:hist2})
  for Algorithm \ref{alg:2} applied to an \(M/M/c\) queue for 5,000
  runs with $\lambda=30$, $\mu=2$ and $c=30$. Chi-squared test on
  distribution of number of customers in system gave a \(p\)-value of
  \(0.84\). (b) Coalescence time (measured as in
  Figure~\ref{fig:hist2}) for Algorithm \ref{alg:2} applied to an
  \(M/M/c\) queue for 5,000 runs with $\lambda=50$, $\mu=2$ and $c=50$. Chi-squared test on distribution of number of customers in system gave a \(p\)-value of \(0.38\).}\label{fig:hist3}
\end{figure}

The run-time of Algorithm \ref{alg:1} is equal to the time taken for the $[M/G/1\;PS]^c$ dominating process $\hat Y$ to empty. It is well known  \cite[Theorem 5.7]{Asmussen-2003} that in equilibrium the mean time for the $M/G/1$ queue to empty is finite if and only if the service duration distribution has finite second moment. Using our standing assumption that $\Expect{S^2}<\infty$, it follows that each server in $\hat Y$ will almost surely empty in finite time, and that the time taken until we see a simultaneous empty period for all $c$ servers (that is, the time taken for the dominating process to completely empty) will itself be finite. Thus Algorithm \ref{alg:1} has finite mean run-time if and only if $\Expect{S^2}<\infty$. The same observation holds for Algorithm \ref{alg:2}, for which a finite run-time is also dependent upon each server in $\hat Y$ emptying in finite time.

Stronger conditions on the moments of $S$ allow a better bound on run-times. For example, note that the time taken for $\hat Y$ to completely empty has an exponential moment (that is, $\hat Y$ is geometrically ergodic) if and only if the $M_{\lambda/c}/G/1$ queue (and indeed the $M/G/c\;[FCFS]$ target process) is geometrically ergodic; this is equivalent to $S$ itself having a finite exponential moment  \cite[Theorem 16.4.1]{Meyn-2009}. More general conditions for existence of moments of the stationary waiting time for a $GI/GI/c$ queue have recently been determined in  \cite{FossKorshunov-2012}.

Existence of an exponential moment is rather a strong demand, but bounds on algorithm run-time can still be produced under weaker drift conditions. In particular, if $\Expect{S^m}<\infty$ for some $m\geq 2$ then Hou and Liu \cite{Hou-2004} show the embedded $M/G/1$ queue (and hence the embedded $M/G/c$ queue) to be polynomially ergodic, which in turn implies that the run-time of Algorithms \ref{alg:1} and \ref{alg:2} will possess a finite $m^{\text{th}}$ moment. (A similar result is true in a much wider context:  Connor and Kendall \cite{ConnorKendall-2007a} describe a generic (but impractical) perfect simulation algorithm for a class of so-called \emph{tame} chains. This extends the work in \cite{Kendall-2004c} to chains which satisfy a geometric Foster-Lyapunov drift condition at a \emph{state-dependent subsampling time}: see  \cite{ConnorFort-2009} for more details. In particular, if the $M/G/c$ queue is tame, with $\Expect{S^m}<\infty$ for some $m\geq 2$, then Proposition 4.3 of  \cite{ConnorFort-2009} can be used to show that the dominating process of  \cite{ConnorKendall-2007a} for the $M/G/c$ queue is polynomially ergodic.)

The mean run-time behaviour of Algorithm \ref{alg:1} can be estimated using simple renewal-theoretic arguments. 
First, consider an \(M/G/1\) queue with service duration distributed as the random variable \(S\), and with arrival intensity \(\lambda/c\).
By Pollaczek-Khintchine theory, in statistical equilibrium the probability of this being empty is \(1-\rho/c\), where \(\rho=\lambda\Expect{S}\)  \cite[Theorem 5.2]{Asmussen-2003}.
Consequently the \([M/G/1]^c\) dominating process, used in both Algorithms \ref{alg:1} and \ref{alg:2}, has probability \((1-\rho/c)^c\) 
of being completely empty at a given time.

Second, consider the start- and end-times of the busy periods of the whole system \([M/G/1]^c\). 
These form an alternating renewal process: completely empty periods have Exponential\((\lambda)\) durations,
while busy periods are distributed as a random variable \(B_1\), being the time it takes for \([M/G/1]^c\) to empty completely if it starts off with just one new customer.
Alternating renewal theory allows us to deduce
\[
 \frac{1/\lambda}{1/\lambda + \Expect{B_1}}\quad=\quad (1-\rho/c)^c \,. 
\]
Now let \(B_e\) be the time till the queue empties, if it is started in equilibrium. 

A stochastic comparison argument shows that \(B_e \) stochastically dominates \(B_1\) except when the equilibrium queue is empty (probability \((1-\rho/c)^c\)).
Consequently we may deduce
\begin{align*}\label{eqn:lower-bound-alg-1}
 \Expect{B_e}\quad\geq\quad 
(1-(1-\rho/c)^c) \Expect{B_1}\quad&=\quad 
(1-(1-\rho/c)^c) \frac{(1-\rho/c)^{-c}-1}{\lambda} \\
&\geq\quad
\frac{(1-\rho/c)^{-c}-2}{\lambda} \,. 
\end{align*}
This carries over to a lower bound on the run-time of Algorithm \ref{alg:1}, which is given by the complete emptying time of an \([M/G/1 \; PS]^c\) system.

It is instructive to consider specific cases: Table \ref{table:1} shows how the lower bound increases quickly with the arrival rate when $\lambda=c = 2\rho$, and indicates the effect of increasing $\rho$ when $\lambda=c$ is held constant. The lower bound when $\lambda=c=10$ and $\rho=5$ is comparable to the simulation results for Algorithm \ref{alg:1} displayed in Figure~\ref{fig:hist2}, for which the mean run-time was 143. Furthermore, this analysis indicates that Algorithm \ref{alg:1} is infeasible for large \(c\), even when $\rho/c$ is significantly smaller than \(1\).

\begin{table}[th]
\begin{center}
\caption{Lower bound on the mean run time for Algorithm \ref{alg:1} ($\Expect{B_e}$) for some specific queue parameters.} 
\begin{tabular}{llll|llll}
\hline
 \(\lambda\) & \(c\)  &\(\rho\) & lower bound on $\Expect{B_e}$  &  \(\lambda\) & \(c\)  &\(\rho\) & lower bound on $\Expect{B_e}$ \\
\hline
 10                & 10           & 5                          & \(102\) & 30                & 30           & 5                          & \(7.85\) \\
 20                & 20           & 10                         & \(52429\) &  30                & 30           & 10                         & \(6392\)\\
 30                & 30           & 15                         & \(3.58 \times 10^7\) & 30                & 30           & 20                         &  \(6.86 \times 10^{12}\) \\
 40                & 40           & 20                         & \(2.75 \times 10^{10}\) & 30                & 30           & 25                         & \(7.37 \times 10^{21}\) \\
 50                & 50           & 25                         & \(2.25 \times 10^{13}\) & 30                & 30           & 29.5                         & \(7.37 \times 10^{51}\)\\
\hline
\end{tabular}
\label{table:1}
\end{center}
\end{table}

A corresponding analysis for Algorithm \ref{alg:2} is more intransigent, as one has to estimate the mean coupling time of upper- and lower-processes which are coupled \(M/G/c\), coupled by having the
same arrival processes and obtaining service durations in the same sequential order from a fixed sequence. 
We are not yet able to give a useful analysis of this coupling, which would involve consideration of the coupled Kiefer-Wolfowitz workload vectors.
Instead we offer an heuristic argument, working instead with the Markov processes given by numbers of customers in system for two coupled \(M/M/c\) queues with the same stable parameters.
These queues \(X\) and \(Y\) are defined as follows: \(X\) is begun at \(X_0\), a draw from the stationary distribution; \(Y\) is begun at \(Y_0=0\).
Both \(X\) and \(Y\) use the same Poisson stream of arrivals.
Departures from \(X\) and \(Y\) are coupled so that \(X\geq Y\) always: any departure from \(Y\) always coincides with a departure from \(X\). 
Thus the continuous-time Markov chains \(X\) and \(Y\) are \emph{immersion coupled} (\cite{Kendall-2013a}; this kind of coupling is also called \emph{Markovian} or \emph{co-adapted}):
their joint transition rates are given by
\begin{align*}
 X \to X+1\,,Y \to Y+1 &\quad \text{at rate }\lambda\,,\\
 X \to X-1\,,Y \to Y-1 &\quad \text{at rate } (Y\wedge c)\mu\,,\\
 X \to X-1\,,Y \to Y   &\quad \text{at rate } ((X\wedge c)-(Y\wedge c))\mu\,.
\end{align*}
We wish to consider \(\Expect{T_\text{couple}}\), where 
\[
 T_\text{couple}\quad=\quad \inf\left\{t: X_t=Y_t\right\}\,.
\]
Note that \(\Prob{T_\text{couple}\leq t}=\Prob{\text{upper- and lower-processes (begun at time \(-t\)) coalesce by time \(0\)}}\).

To this end, we introduce a further process \(Z\geq X\geq Y\), with \(Z_0=X_0\)
and coupled to \(Y\) as follows:
\begin{align*}
 Z \to Z+1\,,Y \to Y+1 &\quad \text{at rate }\lambda\,,\\
 Z \to Z-1\,,Y \to Y-1 &\quad \text{at rate } (Y\wedge c)\mu\,,\\
 Z \to Z-1\,,Y \to Y   &\quad \text{at rate } \mu \quad \text{ when } Y<c \text{ and } Z>Y\,.
\end{align*}
Then \(\widetilde{T}_\text{couple}\) stochastically dominates \(T_\text{couple}\), hence \(\Expect{\widetilde{T}_\text{couple}}\geq \Expect{T_\text{couple}}\), where
\[
 \widetilde{T}_\text{couple}\quad=\quad \inf\left\{t: Z_t=Y_t\right\}\,.
\]

To estimate \(\Expect{\widetilde{T}_\text{couple}}\), it suffices to find positive constants \(\alpha\) and \(\beta\) such that
\[
 U \quad=\quad \alpha (Z-Y) + \beta Y + t
\]
is a non-negative supermartingale up to the coupling time \(\widetilde{T}\). 
For then we can apply the methodology of the proofs of Foster-Lyapunov criteria: \(\Expect{\widetilde{T}_\text{couple}}\leq\Expect{U_{\widetilde{T}_\text{couple}}}\leq \Expect{U_0}=\alpha\Expect{Z_0}=\alpha\Expect{X_0}\), 
which last can be computed using detailed balance  \cite[page 77]{Asmussen-2003}.
Accordingly, consider the transition rates of the Markov chain \((Z,Y)\): using these we may deduce that, before the coupling time,
\[
 \Expect{U_{t+\delta t} | U_t} - U_t \quad=\quad
 -\alpha \mu \Indicator{Y_t<c} \delta t - \beta (\mu(Y_t\wedge c)-\lambda) \delta t + \delta t + o(\delta t)\,.
\]
For \(U\) to be a supermartingale before the coupling time, it is necessary and sufficient that this expression be non-positive for small \(\delta t\). Non-positivity follows if
\begin{align*}
 \beta (c \mu -\lambda ) \quad&\geq\quad 1 \qquad (\text{case } Y_t\geq c)\,,\\
 \alpha \mu + \beta (\mu(Y_t\wedge c)-\lambda) \quad\geq\quad \alpha \mu-\beta \lambda \quad&\geq\quad 1 \qquad (\text{case } Y_t< c)\,.
\end{align*}
Using \(\rho=\lambda/\mu\), we set
\begin{align*}
 \beta  \quad&=\quad \frac{1}{\lambda}\frac{\rho}{c-\rho}\,,\\
 \alpha \quad&=\quad \frac{1}{\lambda}\frac{c\rho}{c-\rho}\,,
\end{align*}
(note that stability of the \(M/G/c\) queues requires \(\rho\leq c\)) and deduce
\begin{equation}\label{eqn:upper-bound-alg-2-heuristic}
 \Expect{T_\text{couple}}\quad\leq\quad \frac{1}{\lambda} \frac{c\rho}{c-\rho} \Expect{X_0}\,.
\end{equation}

Again, it is instructive to consider specific cases: Table \ref{table:2} tabulates corresponding heuristic over-estimates for Algorithm \ref{alg:2}  for the same ranges of parameter values
as found in Table \ref{table:1}. Note that \(\rho=\lambda\Expect{S}\). 
Note too that the large growth in mean run-time at the foot of the second column of Table \ref{table:2} is an inevitable consequence of heavy traffic in the dominating \([M/G/1]^c\) queue.

\begin{table}[ht]
\begin{center}
\caption{Heuristic over-estimate of the mean run time for Algorithm \ref{alg:2} for some specific queue parameters.} 
\begin{tabular}{llll|llll}
\hline
 \(\lambda\) & \(c\)  &\(\rho\) & \emph{heuristic} upper bound&  \(\lambda\) & \(c\)  &\(\rho\) & \emph{heuristic} upper bound \\
& & &   of mean run-time & & & &  of mean run-time \\
\hline
 10                & 10           & 5                          & \(5.04\) & 30 & 30 & 5 & \(1.00\) \\
 20                & 20           & 10                         & \(10.00\)& 30 & 30 & 10 & \(5.00\)  \\
 30                & 30           & 15                         & \(15.00\) & 30 & 30 & 20 & \(40.10\) \\
 40                & 40           & 20                         & \(20.00\) & 30 & 30 & 25 & \(131.25\) \\
 50                & 50           & 25                         & \(25.00\) & 30 & 30 &29.5 & \(4853.97\) \\
\hline
\end{tabular}
\label{table:2}
\end{center}
\end{table}

Compare the results on \(\log_2\) run-times displayed in
Figure~\ref{fig:hist2} and Figure~\ref{fig:hist3} (for which the mean
values of $\hat T$ for the results in Figure~\ref{fig:hist2} ($\lambda=c=10,\, \rho=5$), Figure~\ref{fig:hist3}(a) ($\lambda=c=30,\, \rho=15$) and Figure~\ref{fig:hist3}(b) ($\lambda=c=50,\, \rho=25$) are 2.27, 2.99 and 3.32 respectively).
Bearing in mind the demands of binary back-off,
this suggests that this heuristic over-estimate is a reasonable indication of the feasibility of Algorithm \ref{alg:2} for substantial values of \(c\).

\section{Conclusion}\label{sec:conclusion}

In this paper we have described the construction of two dominated CFTP Algorithms \ref{alg:1} and \ref{alg:2} for a general stable \(M/G/c\; [FCFS]\) queue,
and have shown that the algorithms have finite mean run-time when the typical service duration has finite second moment.

The second of these, Algorithm \ref{alg:2}, is more complex and requires more delicacy and care in description and in implementation; 
however, despite this increased complexity, it demonstrates the potential for considerably reduced actual run-times compared
with the first, more na{\"\i}ve, algorithm. In particular, Algorithm \ref{alg:2} will be preferable in cases when the \(M/G/c\; [FCFS]\) queue
is stable rather than super-stable. This is because Algorithm \ref{alg:1} has run-time comparable to the time at which the queue first empties;
this time may be expected to be large when super-stability fails.

There has been significant contemporary interest in perfect simulation for queueing problems: we have noted the work of  \cite{BlanchetDong-2013}
for $GI/GI/c/c$ \emph{loss} processes, and of  \cite{MousaviGlynn-2013} on a Brownian model for heavy traffic situations. 
 A particular motivation for this is provided in \cite{GuptaHarcholBalterDaiSwart-2009}, where it is established that
a crucial measure of queue efficiency (mean waiting time) can be substantially affected by more subtle features than simply the first two moments of service time and arrival rate.
In such cases it is of clear value to have access to perfect simulation methods such as these two dominated CFTP algorithms.

We close by mentioning four further questions raised by this work:
\begin{enumerate}
\item 
It is natural to ask whether dominated CFTP methods can be extended
to the case of renewal process input; the work of  \cite{BlanchetDong-2013} on the same problem in the context of loss processes may be very helpful here. 
We believe this should be feasible, particularly because the ``impractical'' dominated CFTP algorithms for polynomially ergodic Markov chains address similar problems
 \cite{ConnorKendall-2007a}.
However we have not considered the details of such an extension.
\item
It would be interesting to know whether anything theoretical can be said about the relative merits of the two back-off strategies for Algorithm \ref{alg:2} outlined in Section~\ref{sec:back-off}.
\item
It is natural to ask whether it is possible to implement dominated CFTP for a suitably wide class of queueing networks.
 Sigman \cite{Sigman-2013} describes several applications to networks of the method in \cite{Sigman-2011} for the super-stable case.
It is not clear to us how one would contrive to construct a dominating process for such problems in the stable case, so this question seems to us to be challenging.
\item
Finally one might ask whether in the context of \(M/G/c\; [FCFS]\) it is possible to carry out dominated CFTP simultaneously for a suitable range of \(c\) the number of servers;
what might be described as ``omnithermal dominated CFTP'', to borrow a term used to describe  Grimmett's coupling of random-cluster processes
for all values of a specific parameter \cite{Grimmett-1995}, and applied to CFTP in  \cite{ProppWilson-1996a}. 
In his PhD thesis \cite{Shah-2004}, Shah showed how to implement omnithermal dominated CFTP for area-interaction point processes; 
the issue for stable \(M/G/c\; [FCFS]\) queues is one of detecting at what stage there is coalescence for all \(c\) in the range.
It is straightforward to carry out omnithermal dominated CFTP in the case of Algorithm \ref{alg:1} in a relatively efficient manner: 
once an emptying time \(\tau\) has been established for the instance with lowest \(c\) in range,
then this will serve for all higher values of \(c\), using a simple workload domination argument (see, eg,  \cite{Moyal-2013}).
(Indeed, the value of \(\tau\) can be updated to the most recent emptying time for the current value of \(c\) after the queue has been simulated for that value.)
However it is an open question whether one can establish a comparably efficient omnithermal dominated CFTP based on Algorithm \ref{alg:2}.
\end{enumerate}

\acks
SBC was supported by EPSRC Research Grant EP/J009180; WSK was supported in part by EPSRC Research Grant  EP/K013939.

%
%
%
%

\bibliographystyle{apt}
\bibliography{PerfectQueues}

\end{document}